\documentclass [11pt]{amsart}
\usepackage{amsmath}
\usepackage{amssymb}
\usepackage{amscd}
\usepackage{epsfig}
\usepackage{amsxtra}
\usepackage{pifont}
\usepackage{mathabx}
\usepackage{accents}
\usepackage{scalerel,stackengine}
\usepackage{xcolor}
\usepackage{hyperref}
\usepackage{soul}

\stackMath
\newcommand\reallywidehat[1]{%
\savestack{\tmpbox}{\stretchto{%
  \scaleto{%
    \scalerel*[\widthof{\ensuremath{#1}}]{\kern-.6pt\bigwedge\kern-.6pt}%
    {\rule[-\textheight/2]{1ex}{\textheight}}
  }{\textheight}%
}{0.5ex}}%
\stackon[1pt]{#1}{\tmpbox}%
}

\newcommand\reallywidecheck[1]{%
\savestack{\tmpbox}{\stretchto{%
  \scaleto{%
    \scalerel*[\widthof{\ensuremath{#1}}]{\kern-.6pt\bigwedge\kern-.6pt}%
    {\rule[-\textheight/2]{1ex}{\textheight}}
  }{\textheight}%
}{0.5ex}}%
\stackon[1pt]{#1}{\scalebox{-1}{\tmpbox}}%
}

\newcommand\reallywidetilde[1]{%
\savestack{\tmpbox}{\stretchto{%
  \scaleto{%
    \scalerel*[\widthof{\ensuremath{#1}}]{\kern-.6pt\bigtilde\kern-.6pt}%
    {\rule[-\textheight/2]{1ex}{\textheight}}
  }{\textheight}%
}{0.5ex}}%
\stackon[1pt]{#1}{\scalebox{-1}{\tmpbox}}%
}

\numberwithin{equation}{section}

\usepackage{geometry}
\newcommand{\supp}{\mbox{\rm supp}}
\newcommand{\dens}{\mbox{\rm dens}}
\newcommand{\udens}{\overline{\mbox{\rm dens}}}
\newcommand{\ldens}{\underline{\mbox{\rm dens}}}
\newcommand{\card}{\mbox{\rm card}}
\newcommand{\vol}{\mbox{\rm vol}}
\newcommand{\RR}{{\mathbb R}}

\newcommand{\ZZ}{{\mathbb Z}}
\newcommand{\CC}{{\mathbb C}}

\newcommand{\NN}{\mathbb N}

\newcommand{\XX}{\mathbb X}
\newcommand{\PP}{\mathbb P}
\newcommand{\AAA}{\mathbb A}

\newcommand{\cM}{{\mathcal M}}

\newcommand{\cA}{{\mathcal A}}
\newcommand{\cB}{{\mathcal B}}

\newcommand{\dd}{\mbox{d}}

\newcommand{\Cc}{C_{\mathsf{c}}}

 \newtheorem{theorem}{Theorem}[section]
 \newtheorem{lemma}[theorem]{Lemma}
 \newtheorem{proposition}[theorem]{Proposition}
 \newtheorem{corollary}[theorem]{Corollary}

 \newtheorem{definition}[theorem]{Definition}
 \newtheorem{example}[theorem]{Example}
  \newtheorem{remark}[theorem]{Remark}

\begin{document}
\title{Diffraction of the primes and other sets of zero density}

\author[A. Humeniuk]{Adam Humeniuk}
\address{Department of Mathematics and Computing, Mount Royal University \newline \hspace*{\parindent}  Calgary, Alberta, Canada}
\email{ahumeniuk@mtroyal.ca}

\author[C. Ramsey]{Christopher Ramsey}
\address{Department of Mathematics and Statistics, MacEwan University, \newline
\hspace*{\parindent}  Edmonton, Alberta, Canada}
\email{ramseyc5@macewan.ca}
\urladdr{https://sites.google.com/macewan.ca/chrisramsey/}

\author[N. Strungaru]{Nicolae Strungaru}
\address{Department of Mathematics and Statistics, MacEwan University, \newline
\hspace*{\parindent}  Edmonton, Alberta, Canada, 
and 
\newline \hspace*{\parindent} 
Institute of Mathematics ``Simon Stoilow'', 
Bucharest, Romania}
\email{strungarun@macewan.ca}
\urladdr{https://sites.google.com/macewan.ca/nicolae-strungaru/home}



\begin{abstract}
In this paper, we show that the diffraction of the primes is absolutely continuous, showing no bright spots (Bragg peaks). We introduce the notion of counting diffraction, extending the classical notion of (density) diffraction to sets of density zero. We develop the counting diffraction theory and give many examples of sets of zero density of all possible spectral types.
\end{abstract}

\maketitle

\section{Introduction}

In 2018, much interest was given to a group of papers \cite{ZMT, TZC1, TZC2} that studied the prime numbers through diffraction experiments. They discovered that the primes in certain large intervals possess a pattern of Bragg-like peaks that is evocative of the diffraction pattern one sees in quasicrystals, aperiodic solids with diffraction that is crystal-like. The main conjecture of these papers was that this discovered pattern showed deep structural results about the prime numbers and that these approximate diffraction pictures suggest that the whole set of primes have a pure point diffraction spectrum.

In this paper, we establish in Theorem \ref{thm:main1} that the diffraction measure of the prime numbers is the Lebesgue measure, therefore, absolutely continuous. This means that any perceived Bragg-like peak structure observed in a physical diffraction experiment of a finite portion of the primes is an artifact of the experiment's finiteness and disappears when taking the limit. As always, the effectiveness of an approximation comes down to what form of convergence is being used and how fast that convergence happens. We argue here that the correct setting is the convergence of measures in the vague topology, which is the foundation of mathematical diffraction theory, which was developed with its feet firmly planted in physics \cite{Cowley, Hof1, Dworkin, BL}. It is of note that the diffraction of the finite approximations converge extremely slowly for many highly ordered structures, and are nearly impossible to pick up convincingly in a physical experiment, notably the circular symmetry of the pinwheel tiling (see for example \cite[Section 4]{GD}). As we discuss in Remark~\ref{rem:diff-primes-conv}, this also seems to be the case for the diffraction of the primes, which explains why the simulations in \cite{ZMT, TZC1, TZC2} do not show the real picture for the infinite set of primes.

The diffraction of the mathematical idealization of quasicrystals and other nice structures has been studied so far by using the structure's uniform, positive density \cite{Hof1,BL}. However, the prime numbers have zero density in the positive real line so a new approach is needed. Here, we define the counting diffraction, which is studied through its counting version of the autocorrelation (or two-point correlation) measure. The great advantage to this is that the counting and usual diffraction measures correspond on positive density sets, Theorem~\ref{thm:rel-densn0}, while on sets of zero density, the former measure gives a sensible extension of the Patterson formula to an infinite set, Theorem~\ref{thm:rel-dens0}.

Given a finite set $F \subseteq \RR^d$, its diffraction is given by the Patterson formula:
\[
I_F(y):=\frac{1}{\card(F)} \left| \sum_{x \in F} e^{2 \pi i x \cdot y} \right|^2 \,.
\]
A simple computation shows that $I_F$ is the Fourier transform of the finite measure
\[
\gamma^{}_{F}:= \frac{1}{\card(F)} \delta_{F}\ast\widetilde{\delta_{F}} =\frac{1}{\card(F)} \sum_{x,y \in F} \delta_{x-y} \,.
\]
The measure $\gamma^{}_{F}$ is called the autocorrelation, or 2-point correlation, measure of the finite sample $F$.

Next, consider an infinite uniformly discrete set $\varLambda \subseteq \RR^d$ and some nice averaging sequence $(A_n)_n$ in $\RR^d$, such as $A_n=[-n,n]^d$. The diffraction of $\varLambda$ is defined as the limit in a suitable topology of the diffraction measures $I_{F_n}$ of the finite sets $F_n:= \varLambda \cap A_n$. For uniformly discrete sets $\varLambda$ of positive density, it is more advantageous to define the autocorrelation and diffraction of $\varLambda$ by dividing by the volume, $\vol(A_n)$, of the averaging sequence instead of the cardinality, $\card(F_n)$. 
As introduced by Hof \cite{Hof1} in $\RR^d$, the sequence 
\[
\gamma_n := \frac{1}{\vol(A_n)} \delta_{F_n}\ast\widetilde{\delta_{F_n}}  
\]
of autocorrelation measures of the finite samples $F_n$ has a subsequence converging to some positive definite measure $\gamma$. The measure $\gamma$ is Fourier transformable, and, by \cite[Lemma 4.11.10]{MoSt} or \cite[Theorem~4.5]{BF}, its Fourier transform $\widehat{\gamma}$ is a positive measure called the (density) diffraction of $\varLambda$. 

The choice of averaging by $\vol(A_n)$ has the great advantage that one can often use spectral theory of dynamical systems via the so called Dworkin argument (see \cite{Dworkin,BL,Gou}, just to name a few). The relationship 
\[
\gamma_n= \frac{\card(F_n)}{\vol(A_n)} \gamma^{}_{F^{}_n} 
\]
shows that this change from $\gamma^{}_{F^{}_n}$ to $\gamma_n$ simply multiplies the diffraction measure by the positive density
\[
d= \lim_{n \to \infty} \frac{\card(F_n)}{\vol(A_n)}
\]
of the point set. In particular, for sets of positive density, the two approaches should be equivalent, and they are indeed equivalent (see Theorem~\ref{thm:rel-densn0}). On the other hand, the same relation shows that, for sets of density zero, defining the autocorrelation as the limit of $\gamma_n$ is wrong, as this would always give a zero diffraction measure. 

Over the last few years, there has grown interest in the field of aperiodic order towards the study of point sets that are not relatively dense, such as for example maximal density weak model sets (see \cite{BHS,KR15,KR19}). So far, they have been studied using certain averaging sequences with respect to which they have positive density. It is worth emphasizing that a point set $\varLambda$ is not relatively dense if and only if there exists a van Hove sequence with respect to which the set has a density of zero\footnote{In fact, this is equivalent to the seemingly stronger statement that for each van Hove sequence $(A_n)_n$ there exists a sequence $(t_n)_n$ in $G$ such that $\varLambda$ has zero density with respect to the van Hove sequence $(t_n+A_n)_n$. }. These examples, as well as the diffraction of primes we mention above, suggest that it would be good to properly define the diffraction for sets of density zero, which we do in this paper.


The paper is organized as follows: in Section~\ref{sec:preliminaries} we collect some basic definitions and properties needed in the paper, and define the density and counting autocorrelations and diffractions. We study the existence and basic properties of the autocorrelation/diffraction measures, as well as their relationship in Section~\ref{sect:autrel}. In particular, we show that for uniformly discrete point sets, both the density and counting autocorrelations exist along subsequences of a given van Hove sequence (Proposition~\ref{prop:diffraction_exists}). Furthermore, for sets of positive density, the counting and density diffraction measures are proportional (Theorem~\ref{thm:rel-densn0}), while for sets of density zero, the density diffraction is always null while the counting diffraction is non-zero for non-trivial point sets. In Section~\ref{sect:boring} we show that for sets which are very sparse, such as the Fibonacci numbers or any fast-growing sequence, the diffraction is the Lebesgue measure. In Section~\ref{sect:primes} we prove in Theorem~\ref{thm:main1} that the diffraction of the primes is the Lebesgue measure. This is one of the main results of the paper, and shows that the primes have absolutely continuous diffraction spectrum. We show that this holds when the averaging sequence is any reasonable van Hove sequence of intervals. We also study the diffraction of some related sets, such as the diffraction of prime powers and the diffraction of twin primes. We complete the paper by showing in Section~\ref{Sect:last} that there exist sets of density zero with  diffraction of any spectral type and showing in Section~\ref{Sect:lastlast} that the counting diffraction is preserved when a point set is embedded into higher-dimensional space.

\section{Preliminaries}\label{sec:preliminaries}

We begin with recalling a number of definitions. It should be noted that everything before the definition of counting autocorrelation and counting diffraction for infinite sets is standard.

Throughout this paper, $G$ is a second countable locally compact abelian group (LCAG), and any such group is metrisable. While in most examples $G$ will simply be $\RR$, we want to set up the theory of counting diffraction in general settings.  


\begin{definition}
A set $\varLambda \subseteq G$ is called \textbf{uniformly discrete} if for any choice of metric $d$ that generates the topology on $G$, there exists a constant $r>0$ such that for all  $x,y\in \varLambda$, if $x\ne y$, then $d(x,y)\ge r$.

The set $\varLambda \subseteq G$ is called \textbf{locally finite} if, for all compact sets $K \subseteq G$, the set $\varLambda \cap K$ is finite.  
\end{definition}

It is easy to see that $\varLambda \subseteq G$ is uniformly discrete if and only if there exists some non-empty open set $U$ such that, any translate $t+U$ of $U$ meets $\varLambda$ in at most one point. 
On another hand, $\varLambda$ is locally finite exactly when $\varLambda$ is closed and discrete in $G$.

Finally, a set $\varLambda$ is locally finite exactly when the Dirac comb
\[
\delta_\varLambda :=
\sum_{x\in \varLambda} \delta_x
\]
is a locally finite measure on $G$.

Next, let us recall the definition of FLC.

\begin{definition}
A point set $\varLambda$ has \textbf{finite local complexity} (FLC) if $\varLambda-\varLambda$ is locally finite.
\end{definition}

\smallskip
{
For a finite measure $\mu$ on $G$, 
define the \textbf{autocorrelation measure} of $\mu$ as
\[
\gamma_{\mu}:= 
\left\{
\begin{array}{cc}
\frac{1}{\left|\mu\right|
(G)} \mu \ast \widetilde{\mu}, & \mbox{ if } \mu \neq 0\,, \\
0, & \mbox{ if } \mu =0 \,,
\end{array}
\right. 
\]
where $\ast$ denotes convolution of finite measures, and $\widetilde{\mu}$ is the reflected conjugate measure 
\[
\widetilde{\mu}(\varphi) :=
\overline{\mu(\overline{\varphi})} \qquad \forall \varphi \in \Cc(G) \,.
\]
In particular, if $F \subseteq G$ is a finite set, its autocorrelation is simply the autocorrelation $\gamma^{}_{F}$ of $\delta_{F}$:
\[
\gamma^{}_{F}:=
\left\{
\begin{array}{cc}
\frac{1}{\card(F)} \delta_F\ast \widetilde{\delta_F}, & \mbox{ if } F \neq \emptyset\,, \\
0, & \mbox{ if } F = \emptyset \,.
\end{array}
\right.
\]
Note that when $F \neq \emptyset$ we have 
\[
\gamma^{}_{F} =\frac{1}{\card(F)} \sum_{x,y \in F} \delta_{x-y} \,.
\]
}

Recall that the \textbf{Fourier transform} of a finite Radon measure $\mu$ on a LCAG $G$ is the extended complex function $\widehat{\mu}$ on the dual group $\widehat{G}$ defined by
\[
\widehat{\mu}(\chi) =
\int_G \overline{\chi(\mu)} \; \dd \mu( \chi)
\]
for all $\chi \in \widehat{G}$.
{
\begin{definition}
The \textbf{counting diffraction} of a finite measure $\mu$ is the Fourier transform
\[
\widehat{\gamma_{\mu}}(\chi) =
\left\{
\begin{array}{cc}
\frac{1}{\left| \mu \right|(G)} \widehat{\mu \ast \widetilde{\mu}} (\chi)= \frac{1}{\left| \mu \right|(G)} \left| \widehat{\mu}(\chi) \right|^2  & \mbox{ if } \mu \neq 0  \,, \\
0 & \mbox{ if } \mu=0 \, .
\end{array}
\right.
\]

Similarly, the \textbf{counting diffraction} of a finite set $F\subseteq G$ is the Fourier transform
\[
\widehat{\gamma^{}_{F}}(\chi) =
\left\{
\begin{array}{cc}
\frac{1}{\card(F)} \widehat{\delta_F\ast \widetilde{\delta_F}} (\chi)= \frac{1}{\card(F)} \left| \sum_{x \in F} \chi(x) \right|^2 & \mbox{ if } F \neq \emptyset  \,, \\
0 & \mbox{ if } F= \emptyset \, .
\end{array}
\right.
\]
\end{definition}
}
If $\varLambda$ is not finite, defining the autocorrelation $\gamma_\varLambda$, and therefore the diffraction of $\varLambda$, is more difficult as the convolution $\delta_\varLambda\ast \widetilde{\:\delta_\varLambda\:}$ is ill-defined. If $\varLambda$ is locally finite, then we can avoid this issue by restricting to a sequence $(F_n)_n$ of finite sets $F_n=A_n\cap \varLambda$, where $\cA=(A_n)_n$ is an appropriate averaging sequence of compact sets in $G$. As it was noted by Schlottmann \cite{Sch00}, when dealing with point sets and measures, one has to work with so-called van Hove sequences. Let us recall here the definition.

\begin{definition}
A \textbf{van Hove sequence} $\cA=(A_n)_n$ in a second countable LCAG $G$ is a sequence of compact sets of positive measure such that for every compact subset $K\subseteq G$,
\[
\lim_{n\to \infty} \frac{\vol(\partial^K(A_n))}{\vol(A_n)} = 0,
\]
where ``$\,\vol$" denotes Haar measure, and
\[
\partial^K(A)=((A+K)\setminus A^\circ)\cup((\overline{G\setminus A}-K)\cap A)
\]
is the \textbf{$K$-boundary} of a set $A$.
\end{definition}

The $K$ boundary $\partial^K(A)$ is the set of points which are $K$-translates of points in $A$ that land outside $A^\circ$, or which are points in $A$ which are translated outside of $A^\circ$ by a point in $K$. If $K=\{g\}$ is a singleton, the symmetric difference $(g+A)\triangle A$ satisfies
\[
(g+A)\triangle A \subseteq \partial^{\{g,-g\}}(A),
\]
Therefore if $A_n$ is a van Hove sequence, and $g$ is in $G$, then
\[
\lim_{n\to \infty} \frac{\vol((g+A_n)\triangle A_n)}{\vol(A_n)} = 0.
\]
That is, every van Hove sequence is a F\o lner sequence.

\begin{remark}\
\begin{itemize}
\item[(a)] Every second countable LCAG admits van Hove sequences \cite{Sch00}. 
    \item[(b)]
Some authors define van Hove sequences to be precompact. While this seems to allow more general averaging sequences, it typically makes no difference to calculations. Indeed, it is easy to see that a sequence $(A_n)_n$ of precompact sets is a van Hove sequence if and only if $(\overline{A_n})_n$ is a van Hove sequence. Moreover, \cite[Lemma~1.1]{Sch00} implies that for every translation bounded measure $\mu$ (as defined below) we have 
\[
\lim_{n \to \infty} \left| \frac{1}{\vol(A_n)} \mu(A_n) -\frac{1}{\vol(\overline{A_n})} \mu(\overline{A_n}) \right| =0 \,.
\]
Because of this, restricting to compact van Hove sequences is not a restriction.
\item[(c)] All the examples below are point sets $\varLambda \subseteq \RR$. For those examples, we will always use van Hove sequences $([a_n,b_n])_n$ of intervals. It is easy to see that 
a sequence $([a_n,b_n])_n$ of intervals is a van Hove sequence if and only if $\displaystyle \lim_{n \to \infty} b_n-a_n =\infty$.

\end{itemize}

\end{remark}

We will sometimes prefer to work with van Hove sequences which are \textbf{nested} in the sense that $A_n\subseteq A_{n+1}$ for every $n$. We will also say that $(A_n)_n$ is \textbf{exhausting} if $\bigcup_n A_n = G$. The most typical exhausting van Hove sequences on $\RR^d$ are the cubes $([-n,n]^d)_n$ or balls $(\{x\in \RR^d\;\mid\; \|x\|\le n\})_n$.

\medskip


When proving the existence of the autocorrelation, we will need the following measure theoretic background. 

By the Riesz Representation Theorem \cite{Die,Rud}, a \textbf{Radon measure} (or simply measure) on $G$ can be seen as a linear functional $\mu$ on the space $\Cc(G)$ of compactly supported continuous functions on $G$, which is continuous with respect to the so-called ``inductive" topology on $\Cc(G)$, obtained by writing 
\[
\Cc(G) = \bigcup_{\substack{ K \subseteq G \\ K \mbox{ compact }} } C(G:K) \,,
\]
where
\[
C(G:K):= \{f\in \Cc(G)\;\mid \; \supp(f)\subseteq K\} \,.
\]
The \textbf{vague topology} for measure is defined as the weak-$\ast$ topology on the space $\cM(G)$ of measures on $G$ viewed as the dual space of $\Cc(G)$. More precisely, a net $\mu_\alpha$ of measures converges vaguely to a measure $\mu$ precisely when 
\[
\mu(f) = \lim_\alpha \mu_\alpha(f) \qquad \text{for all } f \in \Cc(G) \,.
\]

\smallskip 

Give a Radon measure $\mu$ on $G$ and a precompact Borel set $B$ we can define 
\[
\|\mu\|_B :=
\sup_{t\in G} |\mu|(t+B) \,,
\]
where $|\mu|$ denotes the total variation measure of $\mu$ (see \cite[Sect.~6.5.5]{Ped} for the definition and properties).

The set of \textbf{translation bounded measures} is defined as
\[
\cM^\infty(G)=\{\mu \in \cM(G)\;\mid \; \|\mu\|_K<\infty \text{ for all compact }K\subseteq G\} \,.
\]
In fact, if $U$ and $V$ are any precompact sets with nonempty interior, then $\|\cdot\|_U$ and $\|\cdot\|_V$ define norms (not just seminorms) on $\cM^\infty(G)$, and these norms are equivalent \cite{BM}.

Let us note here that for any uniformly discrete point set $\varLambda$ we have $\delta_{\varLambda} \in \cM^\infty(G)$ \cite{BL}, which we use below. 

\smallskip 
For any constant $C>0$ and open precompact set $U\subseteq G$, write
\[
\cM_{C,U}(G) :=
\{\mu\in\cM^\infty(G) \;\mid\;
\|\mu\|_U\le C\}.
\]
Since the vague topology is a weak-$\ast$ topology, the Banach-Alaoglu Theorem implies compactness of $\cM_{C,U}(G)$.

\begin{lemma}\cite[Theorem 2]{BL} \label{lem:M_C_U_compact} Let $G$ be a LCAG. Then, the set $\cM_{C,U}(G)$ is vaguely compact. If $G$ is second countable, then the vague topology is metrisable on  $\cM_{C,U}(G)$. \qed 
\end{lemma}

Let us now introduce the following two definitions, and discuss their relevance. Note again that the infinite version of the counting autocorrelation and diffraction are novel definitions and these have not appeared elsewhere to the best of our knowledge.
{
\begin{definition}\label{def:gamma_n} Let $\mu \in \cM^\infty(G)$ be a translation bounded measure, and let $\cA=(A_n)_n$ be a van Hove sequence. Define $\mu_n:=\mu|_{A_n}$ for each $n$, and set
\[
\gamma_{\mu_n} := \left\{ 
\begin{array}{cc}
\frac{1}{\left|\mu_n\right|(G)} \mu_n\ast \widetilde{\mu_n}, & \mbox{ if } \mu_n \neq 0 \,, \\
0, & \mbox{ if } \mu_n =0 \,,
\end{array}
\right.
\]
and
\[
\gamma_{n} :=
\frac{1}{\vol(A_n)}\mu_n\ast \widetilde{\mu_n}  \, ,
\]
where $\vol$ denotes Haar measure.

We say that the \textbf{density autocorrelation} $\gamma_{\operatorname{dens}}$ of $\mu$ exists with respect to $\cA$ if the limit
\[
\gamma_{\operatorname{dens}}=\lim_{n \to \infty} \gamma_n
\]
exists in the vague topology.  
In the same way, we say that the \textbf{counting autocorrelation} $\gamma_{\operatorname{count}}$ of $\mu$ exists with respect to $\cA$ if the limit
\[
\gamma_{\operatorname{count}}=\lim_{n \to \infty} \gamma^{}_{\mu^{}_n}
\]
exists in the vague topology. 

As limits of positive definite measures, $\gamma_{\operatorname{count}}, \gamma_{\operatorname{dens}}$ are positive definite measures \cite[Lemma 4.11.10]{MoSt} and hence Fourier transformable. We refer to the positive measures $\widehat{\gamma_{\operatorname{dens}}}$ and  $\widehat{\gamma_{\operatorname{count}}}$  as the \textbf{density diffraction} and the \textbf{counting diffraction} of $\varLambda$ with respect to $\cA$. 
\end{definition}

\begin{remark}
If $\varLambda \subseteq G$ is a (weakly) uniformly discrete point set, and $\cA=(A_n)_n$ is van Hove sequence, then 
    \begin{align*}
   \gamma_{\operatorname{dens}}&=\lim_{n \to \infty} \gamma_n \\
\gamma_{\operatorname{count}}&=\lim_{n \to \infty} \gamma^{}_{F^{}_n} \,,
    \end{align*}
    where $F_n:=\varLambda \cap A_n$ and
\begin{align*}
\gamma^{}_{F^{}_n} &:= \left\{ 
\begin{array}{cc}
\frac{1}{\card(F_n)}\delta_{F_n}\ast \widetilde{\delta_{F_n}},& \mbox{ if } F_n \neq \emptyset\,, \\
0,&\mbox{ if } F_n =\emptyset\,, \\
\end{array}
\right. \\
\gamma_{n} &:=
\frac{1}{\vol(A_n)}\delta_{F_n}\ast \widetilde{\delta_{F_n}} \,.
\end{align*}
\end{remark}

Let us note here the following simple relation which we will often use:
\begin{equation}\label{eq:density_ratio}
\gamma_n =\frac{\left| \mu \right|(A_n)}{\vol(A_n)} \gamma_{\mu_n} \,.
\end{equation}
In particular, for (weakly) uniformly discrete point sets $\varLambda \subseteq G$ we have
\[
\gamma_n =\frac{\card(\varLambda \cap A_n)}{\vol(A_n)} \gamma^{}_{F^{}_n} \,.
\]
}

\smallskip

Using a standard argument, we explain below in Proposition \ref{prop:diffraction_exists} that if $\varLambda$ is a uniformly discrete point set in $G$, and $\cA=(A_n)_n$ is any van Hove sequence, then there is a subsequence of $\cA$ with respect to which the counting and density autocorrelations of $\varLambda$ both exist.

\begin{remark}\label{rem-supp} {Let $\varLambda \subseteq G$ be a uniformly discrete point set.} Since the supports of each measure satisfy $\supp(\gamma_n), \supp(\gamma^{}_{F^{}_n}) \subseteq \varLambda - \varLambda$ we have 
\[
\supp(\gamma_{\operatorname{dens}}), \,\supp(\gamma_{\operatorname{count}})\ \subseteq \ \overline{\varLambda - \varLambda} \,.
\]
{The same conclusion holds more generally for all weighted combs supported inside $\varLambda$.}
\end{remark}

\smallskip




In the next section, we will show that the density and counting diffractions are related via a relation of the type
\[
\widehat{\gamma_{\operatorname{dens}}} = \dens(\varLambda) \widehat{\gamma_{\operatorname{count}}} \,.
\]
In particular, in Theorem~\ref{thm:rel-densn0} and Theorem~\ref{thm:rel-dens0} we will show the following:

\begin{itemize}
\item{} For point sets of positive density, the counting and density diffraction coincide up to multiplication by non-zero constants. In this case, the two approaches to the diffraction spectrum are equivalent, and either one can be used.
\item{} For point sets of zero density, the density diffraction is zero and the counting diffraction is non-zero. In this case, the counting diffraction must be used.
\end{itemize}

\section{Density and autocorrelation}\label{sect:autrel}

In this section, we study the basic properties of the density and counting autocorrelations. 

\subsection{Density and counting autocorrelations}

\begin{definition} Let $\varLambda \subseteq G$ be uniformly discrete and $\cA = (A_n)_n$ a van Hove sequence. We define the \textbf{lower} and \textbf{upper density of } $\varLambda$ \textbf{with respect to} $\cA$ by 
\begin{align*}
\ldens_{\cA}(\varLambda)&:= \liminf_{n\to\infty} \frac{\card(\varLambda \cap A_n)}{\vol(A_n)} \quad\text{and} \\
\udens_{\cA}(\varLambda)&:= \limsup_{n \to \infty} \frac{\card(\varLambda \cap A_n)}{\vol(A_n)},
\end{align*}
respectively. We say that $\varLambda$ has \textbf{positive density} with respect to $\cA$ if
\[
\ldens_{\cA}(\varLambda)>0 \,.
\]
We say that the \textbf{density of $\varLambda$ exists with respect to $\cA$} if the limit
\[
\dens_{\cA}(\varLambda):= \lim_{n \to \infty} \frac{\card(\varLambda \cap A_n)}{\vol(A_n)} 
\quad\text{exists.}
\]
\end{definition}

By definition, the density of $\varLambda$ exists with respect to $\cA$ if and only if $\ldens_{\cA}(\varLambda)=\udens_{\cA}(\varLambda)$, and if this is the case, then we have
\[
\ldens_{\cA}(\varLambda)=\udens_{\cA}(\varLambda)=\dens_{\cA}(\varLambda) \,.
\]

{
Let us next recall the Besicovitch seminorm for measures.
\begin{definition} For a measure $\mu$ on $G$ and a van Hove sequence $\cA = (A_n)_n$ we define the \textbf{Besicovitch seminorm} $\| \mu\|_{b,\cA}$ via
\[
\| \mu \|_{b,\cA} := \limsup_{n \to \infty} \frac{1}{\vol(A_n)} \left|\mu\right|(A_n) \,.
\]
We say that the \textbf{mean} of $\mu$ exists with respect to $\cA$ if the following limit exists: 
\[
M_{\cA}(\mu):= \lim_{n \to \infty} \frac{\mu(A_n)}{\vol(A_n)} \,.
\]
\end{definition}

It is immediate that for a point set $\varLambda \subseteq G$ and a van Hove sequence $\cA$ we have 
\[
\overline{\dens}_{\cA}(\varLambda) = \| \delta_{\varLambda} \|_{b,\cA} \,.
\]
Furthermore, the density of $\varLambda$ exists if and only if the mean of $\delta_{\varLambda}$ exists, and they are equal in this case.

Finally, whenever the mean of $\mu$ exists with respect to $\cA$ we have \[
\left| M_\cA(\mu) \right| \leq \| \mu \|_{b,\cA} \,.
\]

The following result is a trivial consequence of \cite[Lemma~1.1(2)]{Sch00}.

\begin{lemma}\label{lem:finite_density}
If $\mu$ is translation bounded, then the sequence $\left(\frac{\left| \mu \right|(A_n)}{\vol(A_n)}\right)_n$ is bounded. Furthermore,
\[
\| \mu \|_{b,\cA}< \infty \,.
\]
In particular, for an uniformly discrete set $\varLambda \subseteq G$ we have
\[
\udens_{\cA}(\varLambda) < \infty \,.
\]\qed
\end{lemma}

The following is an immediate consequence.
\begin{corollary} Let $\cA$ be a van Hove sequence.
\begin{itemize}
    \item[(a)] If $\mu \in \cM^\infty(G)$ then the means of $\mu$ and $| \mu|$ exist along a subsequence $\cB$ of $\cA$.    
    \item[(b)] If $\varLambda \subseteq G$ is uniformly discrete, then the density of $\varLambda$ exists along a subsequence $\cB$ of $\cA$.
\end{itemize}\qed
\end{corollary}


As promised in Section \ref{sec:preliminaries}, we can now prove that the counting and density diffractions exist for a translation bounded measure and any van Hove sequence, perhaps upon passing to a subsequence. In particular they exist for uniformly discrete point sets.

The following lemma is standard, but we include the proof for completeness.

\begin{proposition}\label{prop:diffraction_exists}
Let $\mu \in \cM^\infty(G)$ and let $\cA = (A_n)_n$ be a van Hove sequence. Then, for each open precompact set $U$, there exists some $C>0$ such that
\[
\gamma_n, \gamma_{\mu_n} \in \cM_{C,U}(G)\,, \qquad \text{for all } n\,.
\]
In particular, there exists a subsequence $\cB$ of $\cA$ such that $\gamma_{\operatorname{dens}}, \gamma_{\operatorname{count}}, M_{\cB}(\mu), M_{\cB}(|\mu|)$ exist and 
\[
\gamma_{\operatorname{dens}} = M_{\cB}(|\mu|) \gamma_{\operatorname{count}} \,.
\]
Furthermore, the conclusion holds whenever $\varLambda$ is uniformly discrete.
\end{proposition}

\begin{proof}


Pick any open precompact set $V$ such that $\overline{U} \subseteq V$. Since $\nu_n:= \widetilde{\mu_n}/\left|\mu\right|(A_n)$ is a finite measure such that $\left| \nu_n \right|(G)=1$, by \cite[Lemma 6.1]{NS20} we have 
\[
\| \gamma_{\mu_n} \|_{U} \leq \| \mu_n\|_V \left| \nu_n \right|(G) \leq  \| \mu \|_V =: C_1
\]
where $C_1< \infty$ since $\mu \in \cM^\infty(G)$.

Next, since $\frac{\left| \mu \right|(A_n)}{\vol(A_n)}$ is bounded above by some constant $D$, from Lemma~\ref{lem:finite_density}, by \eqref{eq:density_ratio} we have 
\[
\| \gamma_{n} \|_{U} = \left\| \frac{\mu(A_n)}{\vol(A_n)} \gamma_{\mu_n} \right\|_{U} \leq D C_1 \,.
\]
The claim follows by picking $C= \max\{ C_1, DC_1 \}$.

The existence of the limit along subsequences follows now from \ref{lem:M_C_U_compact} and the boundedness of $\frac{\card(F_n)}{\vol(A_n)}$.

Finally, by definition, $\varLambda$ is uniformly discrete if and only if there exists a precompact open set $W$ such that
\[
\| \delta_{\varLambda} \|_{W}  \leq 1 \,.
\]
The claim follows immediately.
\end{proof}

\begin{remark} Let $\mu \in \cM^\infty(G)$ and let $\cA = (A_n)_n$ be a van Hove sequence. Then, any counting autocorrelation $\gamma_{\operatorname{count}}$ along subsequences of $\cA$ is positive definite and translation bounded. Therefore, by \cite[Thm.~4.10.10 and Thm.~ 4.10.12]{MoSt},  there exists positive definite measures $(\gamma_{\operatorname{count}})_{\operatorname{s}}$ and $(\gamma_{\operatorname{count}})_0$ such that the following Eberlein decomposition holds:
\begin{align*}
\gamma_{\operatorname{count}}&=(\gamma_{\operatorname{count}})_{\operatorname{s}} +(\gamma_{\operatorname{count}})_0 \\
\reallywidehat{(\gamma_{\operatorname{count}})_{{\operatorname{s}}}}&= \left(\widehat{\gamma_{\operatorname{count}}}\right)_{{\operatorname{pp}}} \\
\reallywidehat{(\gamma_{\operatorname{count}})_{0}}&= \left(\widehat{\gamma_{\operatorname{count}}}\right)_{\operatorname{c}} \,.
\end{align*}
Here, $\left(\widehat{\gamma_{\operatorname{count}}}\right)_{{\operatorname{pp}}}$ and $\left(\widehat{\gamma_{\operatorname{count}}}\right)_{{\operatorname{c}}}$ denote the pure point and continuous components of the counting diffraction.

Furthermore, if $\supp(\mu)$ is a subset of a Meyer set, the refined Eberlein decomposition of $\gamma_{\operatorname{count}}$ exists by \cite{NS20}.
\end{remark}
}

Let us now look at an example which shows why translation boundedness is an important assumption, even for sets of density zero. We construct a (non-uniformly discrete) set $\varLambda \subseteq \RR$ with zero density such that $\delta_{\varLambda} \notin \cM^\infty(\RR)$ and $\gamma^{}_{F^{}_n}$ has no vaguely convergent subsequence.

\begin{example} 
Define 
\begin{align*}
\varLambda_m &:= \left\{ 4^m - \frac{j}{2^{m+1}}: 1 \leq j \leq 2^m \right\}, \qquad  \qquad \varLambda :=  \bigcup_{m=1}^\infty\: \varLambda_m \,.
\end{align*}
Let us note here in passing that we start at $4^m - \frac{1}{2^{m+1}}$ and not $4^m$ to avoid having integers in $\varLambda_m$, as, in that case, the $F_n$ we introduce below would have an extra point whenever $n=4^m$. This extra point would not change any conclusion, but would make the proof slightly more technical.

By construction we have $\card(\varLambda_m)=2^m$ and $\varLambda_m \subseteq [4^m-\frac{1}{2}, 4^{m})$. This immediately implies that 
\[
F_n= \varLambda \cap [-n,n] = \bigcup_{j=1}^m \varLambda_j
\]
where $m= \lfloor \log_4n \rfloor$ is the last positive integer satisfying $4^m  \leq n$. Note here we are using $(A_n)_n = ([-n,n])_n$ as our usual averaging sequence. In particular, 
\[
\card(F_n)=2 + \cdots + 2^m = 2^{m+1} - 2 = 2(2^{ \lfloor \log_4n \rfloor}-1) \,.
\]
This implies that 
\[
\dens(\varLambda)= \lim_{n \to \infty} \frac{\card(F_n)}{\vol(A_n)} = \lim_{n \to \infty}  \frac{2(2^{ \lfloor \log_4n \rfloor}-1)}{2n} \leq \lim_{n \to \infty} \frac{2^{ \log_4 n}-1}{n} = \lim_{n \to \infty} \frac{\sqrt{n}-1}{n}=0 
\]
as claimed.

Now, fix some non-negative function $f \in \Cc(\RR)$ so that $f(x) \geq 1$ for all $ x \in [-
\frac{1}{2},\frac{1}{2}]$. 
Let $n \in \NN$ be arbitrary, and set 
\[
m={\lfloor \log_4n \rfloor} \,.
\]
Then,
\begin{align*}
\gamma^{}_{F^{}_n}(f) & = \frac{1}{2(2^{\lfloor \log_4n \rfloor}-1)} \sum_{x,y \in F_n} f(x-y) \geq \frac{1}{2(2^{\log_4n})} \sum_{x,y \in \varLambda_m } f(x-y) \,.
\end{align*}
Recall that $\varLambda_m \subseteq [4^m-\frac{1}{2}, 4^{m})$. Therefore, for all $x,y \in \varLambda_m$ we have 
 $x-y \in [-\frac{1}{2}, \frac{1}{2}]$ and hence $f(x-y) \geq 1$. It follows that 
\begin{align*}
\gamma^{}_{F^{}_n}(f) & \geq \frac{1}{2\sqrt{n}}  \sum_{x,y \in \varLambda_m } 1 = \frac{1}{2\sqrt{n}} \left(\card(\varLambda_m)\right)^2 \\
&= \frac{1}{2\sqrt{n}} 2^{2m} =  \frac{1}{2\sqrt{n}} 4^{\lfloor \log_4n \rfloor} \geq \frac{1}{2\sqrt{n}} 4^{ \log_4n -1}=\frac{1}{8} \sqrt{n} \,.
\end{align*}
Therefore, 
\[
\gamma^{}_{F^{}_n}(f) \geq \frac{1}{8} \sqrt{n} \qquad \text{ for all } n,
\]
which implies that $\gamma^{}_{F^{}_n}$ has no vaguely convergent subsequence.

\end{example}

Let us next prove the following result. { Item (a) below was implicitly proven in \cite[Prop.~5]{BHS}.}

\begin{lemma}\label{lem:ac-at-0} Let $\varLambda \subseteq G$ be uniformly discrete. 
\begin{itemize}
    \item[(a)] If $\gamma_{\operatorname{dens}}$ exists with respect to $\cA =(A_n)_n$, then $\dens_{\cA}(\varLambda)$ exists and 
    \[
    \gamma_{\operatorname{dens}}(\{0\})= \dens_{\cA}(\varLambda) \,.
    \]
    \item [(b)] If $\gamma_{\operatorname{count}}$ exists with respect to $\cA$, then the following are equivalent: \medskip
    \begin{itemize}
        \item[(i)] 
        $
            \gamma_{\operatorname{count}}(\{0\})=1.
        $
        \item[(ii)] If $F_n= \varLambda \cap A_n$, then $(F_n)_n$ contains a subsequence of non-empty sets.
        \item[(iii)] There exists some $N$ such that for all $n >N$ we have $F_n \neq \emptyset$.
        \item[(iv)]
        $
            \gamma_{\operatorname{count}}\neq 0.
        $
    \end{itemize}
\end{itemize}
\end{lemma}
\begin{proof}
First, let us note that since $\varLambda$ is uniformly discrete, $0$ is an isolated point of $\overline{\varLambda-\varLambda}$. Therefore, there exists some open set $U\subseteq G$ containing $0$ such that 
\[
(\overline{\varLambda - \varLambda}) \cap U = \{ 0 \} \,.
\] 
For the entire proof, we fix some $f \in \Cc(G)$ such that $f(0)=1, 0 \leq f \leq 1$, and $\supp(f) \subseteq U $.

\bigskip

\textbf{(a)}  
 Then, Remark~\ref{rem-supp} implies that 
\begin{align*}
\gamma(f) &= f(0) \gamma(\{ 0 \})= \gamma(\{ 0\}) \\
\gamma_n(f) &= f(0) \gamma_n(\{ 0 \})= \frac{\card(F_n)}{\vol(A_n)} \,.
\end{align*}
Therefore,
\[
\gamma(\{ 0\})=\gamma(f)=\lim_{n \to \infty} \gamma_n(f) \,,
\]
and so (a) follows.

\bigskip

\textbf{(b)} Using Remark~\ref{rem-supp} again, we have 
\begin{align*}
\gamma^{}_{F^{}_n}(f) &= \gamma^{}_{F^{}_n}(\{0\})= 
\left\{ \begin{array}{cc}
     1, & \mbox{ if } F_n \neq \emptyset\,,  \\
    0,  &  \mbox{ if }  F_n = \emptyset \,.
\end{array}
\right. \\
\gamma_{\operatorname{count}}(f) &=\gamma_{\operatorname{count}}(\{ 0\}) \,.
\end{align*}
Since $\gamma^{}_{F^{}_n}$ converges vaguely to $\gamma_{\operatorname{count}}$, we get 
\[
\gamma_{\operatorname{count}}(\{ 0\})=\gamma_{\operatorname{count}}(f) = \lim_{n \to \infty}  \gamma^{}_{F^{}_n}(f) = \lim_{n \to \infty} 
\left\{ \begin{array}{cc}
     1, & \mbox{ if } F_n \neq \emptyset\,,  \\
    0,  &  \mbox{ if }  F_n = \emptyset \,.
\end{array} 
\right.
\]
Now, the implications
\textbf{(iii) $\Longleftrightarrow$ (ii) $\Longleftrightarrow$ (i) $\Longrightarrow$ (iv)} are immediate. We will complete the proof by showing that \textbf{(iv) $\Longrightarrow$ (iii)}. If we assume by contradiction that (iii) fails, then we can find a subsequence $(F_{k_n})_n$ such that $F_{k_n}=\emptyset$.
But then
\[
\gamma_{\operatorname{count}}=\lim_{n \to \infty} \gamma^{}_{F^{}_n}= \lim_{n \to \infty} \gamma^{}_{F_{k_n}}= \lim_{n \to \infty} \gamma_{\emptyset}= \lim_{n \to \infty} 0 =0 , 
\]
a contradiction. \qedhere
\end{proof} %

{ Lemma~\ref{lem:ac-at-0} does not seem to have an equivalent for translation bounded measures. For weighted Dirac combs $\mu=\sum_{x \in G}  \mu(\{x\}) \delta_x$ with uniformly discrete support the proof of Lemma~\ref{lem:ac-at-0} can be modified to show that  
\begin{align*}
\gamma_{\operatorname{dens}}(\{0\})&=M_{\cA}(|\mu|^2), \\
\gamma_{\operatorname{count}}(\{0\}) &=\lim_n  \frac{ \left| \mu \right|^2(A_n)} {\left| \mu \right|(A_n)}  \,,
\end{align*}
where 
\[
|\mu|^2:= \sum_{x \in G} \left| \mu(\{x\}) \right|^2 \delta_x \,.
\]
In particular, if $|\mu|$ has a positive mean, then 
\[
\gamma_{\operatorname{count}}(\{0\}) =\lim_n  \frac{M_{\cA}(|\mu|^2)}{M_{\cA}(|\mu|)} \,.
\]

Beyond measures with uniformly discrete we cannot say too much. For example, when $\mu = \lambda$ is the Lebesgue measure on $\RR$, we have for all van Hove sequences $\cA$ that:
\begin{align*}
\gamma_{\operatorname{dens}}&=\gamma_{\operatorname{count}}=\lambda \,, \\
\gamma_{\operatorname{dens}}(\{0\}) &= \gamma_{\operatorname{count}}(\{0\}) =0 \,.\\
M_{\cA}(\mu)&=1 \,.
\end{align*}

}

We can now prove that for { translation bounded measures with positive absolute value mean} there is no difference between working with density or counting autocorrelation. { In particular, this holds for sets of positive density.} The same is not true in the case of zero-density sets. 

{ We start by proving the result for translation bounded measures, and then look at uniformly discrete point sets. 

\begin{theorem} Let $\mu \in \cM^\infty(G)$ and $\cA = (A_n)_n$ be a van Hove sequence. Assume that 
\[
\liminf_{n\to\infty} \frac{|\mu|(A_n)}{\vol(A_n)} >0 \,.
\]
\begin{itemize}
\item[(a)] If $\gamma_{\operatorname{count}}$ is any cluster point of $\gamma_{\mu_n}$ then there exists some $C \in (0, \infty)$ and a cluster point $\gamma_{\operatorname{dens}}$ of $\gamma_n$ such that 
\[
\gamma_{\operatorname{dens}}=C \gamma_{\operatorname{count}} \,.
\]
\item[(b)] If $\gamma_{\operatorname{dens}}$ is any cluster point of $\gamma_{n}$ then there exists some $D \in (0, \infty)$ and a cluster point $\gamma_{\operatorname{count}}$ of $\gamma_{\mu_n}$ such that 
\[
\gamma_{\operatorname{count}}=D \gamma_{\operatorname{dens}} \,.
\]
\end{itemize}
\end{theorem} 
\begin{proof}
\textbf{(a)} Let $\cA'$ be the subsequence of $\cA$ along which  $\gamma_{\operatorname{count}}$ is the limit. By Lemma~\ref{prop:diffraction_exists} there exists a subsequence $\cB$ of $\cA'$ along which the density autocorrelation of $\mu$ and the mean $M_{\cB}(|\mu|)$ exist and 
\[
\gamma_{\operatorname{dens}}=C \gamma_{\operatorname{count}} \,
\]
for $C=M_{\cB}(|\mu|) >0$. 

\bigskip

\textbf{(b)} The proof of (b) is a symmetrical argument, in which $D=1/M_{\cB}(|\mu|)$.
\end{proof}

In particular, for point sets, we get:
}
\begin{theorem}\label{thm:rel-densn0} Let $\varLambda \subseteq G$ be uniformly discrete and $\cA = (A_n)_n$ be a van Hove sequence. Suppose $\varLambda$ has a positive density with respect to $\cA$.
\begin{itemize}
\item[(a)] If $\gamma_{\operatorname{count}}$ is any cluster point of $\gamma^{}_{F^{}_n}$ then there exists some $C \in (0, \infty)$ and a cluster point $\gamma_{\operatorname{dens}}$ of $\gamma_n$ such that 
\[
\gamma_{\operatorname{dens}}=C \gamma_{\operatorname{count}} \,.
\]
\item[(b)] If $\gamma_{\operatorname{dens}}$ is any cluster point of $\gamma_{n}$ then there exists some $D \in (0, \infty)$ and a cluster point $\gamma_{\operatorname{count}}$ of $\gamma^{}_{F^{}_n}$ such that 
\[
\gamma_{\operatorname{count}}=D \gamma_{\operatorname{dens}} \,.
\]
\end{itemize}\qed
\end{theorem} 



\begin{remark} \
The relationship between the cluster point(s) $\gamma_{\operatorname{dens}}$ and the cluster point(s) $\gamma_{\operatorname{count}}$ is explained by \eqref{eq:density_ratio}:
\[
\gamma_n= \frac{\card(F_n)}{\vol(A_n)} \gamma^{}_{F^{}_n} \,.
\]
This relationship implies that, as long as $\frac{\card(F_n)}{\vol(A_n)}$ is bounded away from $0$, if two of $\gamma_n, \frac{\card(F_n)}{\vol(A_n)}, \gamma^{}_{F^{}_n}$ exist along $\cA$ then so does the third. Furthermore, by Lemma~\ref{lem:ac-at-0}, if $\gamma_n$ is convergent, then so is  $\frac{\card(F_n)}{\vol(A_n)}$.

This means that for point sets of positive density we have the following implications:

\begin{itemize}
    \item[(a)] If $\gamma_{\operatorname{dens}}$ exists along $\cA$ then, $\dens_{\cA}(\varLambda)$ and $\gamma_{\operatorname{count}}$ exist along $\cA$.
    \item[(b)] Assume that $\gamma_{\operatorname{count}}$ exists along $\cA$.  Then, $\gamma_{\operatorname{dens}}$ exists along $\cA$ if and only if $\dens_{\cA}(\varLambda)$ exists along $\cA$.
\end{itemize}
\end{remark}

In contrast, { for point sets of zero density,} the following result shows that $\gamma_{\operatorname{dens}}$ and $\gamma_{\operatorname{count}}$ are not proportional.

\begin{theorem}\label{thm:rel-dens0} Let $\varLambda \subseteq G$ be uniformly discrete and $\cA=(A_n)_n$ be a van Hove sequence for which $F_n=\varLambda \cap A_n$ are not eventually all empty. Suppose $\dens_\cA(\varLambda)=0$.
\begin{itemize}
\item[(a)] The autocorrelation $\gamma_{\operatorname{dens}}$ of $\varLambda$ exists with respect to $\cA$ and 
\[
\gamma_{\operatorname{dens}}=0 \,.
\]
\item[(b)] If $\gamma_{\operatorname{count}}$ is any cluster point of $\gamma^{}_{F^{}_n}$ then 
\[
\gamma_{\operatorname{count}} \neq 0 \,.
\]
\end{itemize}
\end{theorem} 
\begin{proof}
\textbf{(a)} By Lemma~\ref{prop:diffraction_exists} we have 
\[
\gamma_n\in \cM_{C,U}(G)
\]
for some $C>0$ and precompact open $U$. Since $\cM_{C,U}(G)$ is vaguely compact and metrisable, showing that $\gamma_n \to 0$ is equivalent to showing that if $\eta$ is a cluster point of $\gamma_n$ then $\eta =0$.

Let $\eta$ be a cluster point of $\gamma_n$ calculated along some subsequence $\cA'$ of $\cA$. By Lemma~\ref{prop:diffraction_exists}, there exists some subsequence $\cB$ of $\cA'$  such that 
\[
\eta = \dens_{\cB}(\varLambda) \gamma_{\operatorname{count}} \,.
\]
The zero density property gives $\dens_{\cB}(\varLambda)=0$ and hence $\eta=0$. This proves (a).

\textbf{(b)} By Lemma \ref{lem:ac-at-0} and (a),
\[
\gamma_{\operatorname{count}}(\{0\})=1 \neq 0 = \gamma_{\operatorname{dens}}(\{0\}) \,.\qedhere
\]
\end{proof}

{ If $\mu$ is a translation bounded measure such that $M_{\cA}(|\mu|)=0$, then the same proof as Theorem~\ref{thm:rel-dens0} shows that $\gamma_{\operatorname{dens}}=0$.

The equivalent statement for Theorem~\ref{thm:rel-dens0} (b) does not hold for arbitrary measures. In fact it does not hold even for positive weighted Dirac combs with uniformly discrete support. 

Below we give an example of an infinite positive measure $\mu$ with lattice support for which  $\gamma_{\operatorname{count}}=0$.

\begin{example} Let $(a_n)_n \in \ell^2 \backslash \ell^1$, such as for example $a_n=\frac{1}{n}$. Set
\[
\mu:= \sum_{n \in \NN} a_n \delta_n \,.
\]
Let $(A_n)_n =([-n,n])_n$. Then, 
\[
|\mu_n|(A_n)=\sum_{k=1}^n |a_k| \to \infty \,,
\]
as $(a_n)_n \notin \ell^1$. In particular, $|\mu|(\RR)=\infty$.

Now, a simple computation shows that 
\[
\gamma_{\mu_n}(\{0\}) =\frac{\sum_{k=1}^n |a_k|^2}{\sum_{k=1}^n |a_k|} \to 0 \,.
\]
Since $\gamma_{\mu_n}$ is positive definite and supported inside $\ZZ$ we have 
\[
\left| \gamma_{\mu_n}(\{k\}) \right| \leq \gamma_{\mu_n}(\{0\}) \,. 
\]
This immediately implies that $\gamma_{\mu_n}$ converges vaguely to $0$. Thus
\[
\gamma_{\operatorname{count}}=0 \,.
\]
\end{example}
\bigskip

For the rest of the paper we will restrict to the case of point sets.
}

\bigskip

\subsection{Point sets with finite local complexity}
 
If $\varLambda$ has finite local complexity (FLC), then the following lemma shows how to write the autocorrelation as a Dirac comb. { For weighted Dirac combs supported inside Meyer sets the result below was proven in \cite{BM}.}

\begin{lemma}\label{lem:gammadens} Let $\varLambda \subseteq G$ have FLC and let $\cA = (A_n)_n$ be a van Hove sequence. Then $\gamma_{\operatorname{dens}}$ exists if and only if for all $t \in \varLambda - \varLambda$ the limit
\[
\eta_{\operatorname{dens}}(t):= \lim_{n \to \infty} \frac{1}{\vol(A_n)} \card(F_n\cap (F_n-t)) = \dens_{\cA}(\varLambda \cap (\varLambda-t))
\]
exists. Moreover, in this case we have 
\[
\gamma_{\operatorname{dens}}= \sum_{t \in \varLambda - \varLambda} \eta_{\operatorname{dens}}(t) \delta_t \,.
\]
\end{lemma}

\begin{proof}
Let $f\in C_c(G)$, and then for any finite $n$, rearranging the sum shows that
\begin{align*}
\gamma_n(f) &=
\frac{1}{\vol(A_n)}\sum_{x,y\in F_n} f(y-x) \\ &=
\frac{1}{\vol(A_n)}\sum_{t\in \varLambda - \varLambda}\sum_{\substack{x\in F_n \\ x+t\in F_n}} f(t) \\ &=
\sum_{t\in \varLambda} \left(\frac{\card(F_n\cap (F_n-t))}{\vol(A_n)}\right)f(t) \\ &=
\left(\sum_{t\in \varLambda} \left(\frac{\card(F_n\cap (F_n-t))}{\vol(A_n)}\right)\delta_t\right)(f),
\end{align*}
where this sum contains at most finitely many nonzero terms, since $f$ has compact support and $\varLambda-\varLambda$ is locally finite.

If each limit
\[
\eta_{\operatorname{dens}}(t) =
\lim_{n \to \infty} \frac{\card(F_n\cap (F_n-t))}{\vol(A_n)}
\]
exists, then this shows that $\gamma_n$ vaguely converges to
\[
\gamma_{\operatorname{dens}}= 
\sum_{t\in \varLambda-\varLambda} \eta_{\operatorname{dens}}(t)\delta_t.
\]
Conversely, suppose the density diffraction $\gamma_{\operatorname{dens}}=\lim_{n \to \infty} \gamma_n$ exists. Let $t\in \varLambda-\varLambda$. Since $\varLambda-\varLambda$ is locally finite, we can find a function $f\in C_c(G)$ for which $f(t)=1$ and $f(s)=0$ for all other $s\in \varLambda-\varLambda$. In this case, we have
\[
\frac{\card(F_n\cap (F_n-t))}{\vol(A_n)} =
\gamma_n(f),
\]
which converges to $\gamma_{\operatorname{dens}}(f)$, showing that the limit $\eta_{\operatorname{dens}}(t)$ exists.

Finally, it remains to show that $\eta_{\operatorname{dens}}(t)=\dens_\cA(\varLambda\cap (\varLambda -t))$, i.e.
\begin{equation}\label{eq:two_limits_Fn}
\lim_{n \to \infty} \frac{\card(F_n\cap (F_n-t))}{\vol(A_n)} =
\lim_{n \to \infty} \frac{\card(\varLambda \cap (\varLambda -t)\cap A_n)}{\vol(A_n)},
\end{equation}
when either limit exists. Here, 
\[
F_n\cap (F_n-t) \subseteq \varLambda \cap (\varLambda - t)\cap A_n,
\]
with set difference
\begin{align*}
\left(\varLambda \cap (\varLambda - t)\cap A_n\right)
\setminus \left(F_n \cap(F_n-t)\right)  &\subseteq
\varLambda \cap (A_n\setminus  (A_n-t)).
\end{align*}
Therefore, we can bound
\begin{align*}
\left|\frac{\card(\varLambda \cap (\varLambda -t)\cap A_n)}{\vol(A_n)}-\frac{\card(F_n\cap (F_n-t))}{\vol(A_n)}\right| &\le
\frac{\card(\varLambda \cap (A_n\setminus (A_n-t)))}{\vol(A_n)} \\ &=
\frac{\delta_\varLambda(A_n\setminus (A_n-t))}{\vol(A_n)} \leq 
\frac{\delta_\varLambda(\partial^{ \{ t, -t\}} (A_n))}{\vol(A_n)} \to 0\\ 
\end{align*}
by \cite[Lemma~1.1]{Sch00}.  
The claim follows.
\end{proof}

The next result is proven in exactly the same way, and so we will skip the proof.

\begin{lemma}\label{lem:gammacount} Let $\varLambda \subseteq G$ have FLC and let $\cA = (A_n)_n$ be a van Hove sequence such that $F_n=A_n\cap \varLambda$ is eventually always nonempty. Then $\gamma_{\operatorname{count}}$ exists if and only if for all $t \in \varLambda - \varLambda$ the limit
\[
\eta_{\operatorname{count}}(t):= \lim_{n \to \infty} \frac{ \card \left( F_n \cap (-t+F_n) \right)}{ \card \left(F_n \right)}
\]
exists. Moreover, in this case we have 
\[
\gamma_{\operatorname{count}}= \sum_{t \in \varLambda - \varLambda} \eta_{\operatorname{count}}(t) \delta_t \,.
\]\qed
\end{lemma}

Comparing to Lemma \ref{lem:gammadens}, we may not be able to say that the limit $\eta_{\operatorname{count}}(t)$ is the same as the limit
\[
\lim_{n \to \infty}\frac{\card(\varLambda\cap(-t+\varLambda)\cap A_n)}{\card(F_n)}
\]
for general $\varLambda$ and $\cA$. But for van Hove sequences of intervals in $\RR$, and non-trivial uniformly discrete sets $\varLambda \subseteq \RR$, we can do so. { The result below can be seen as a similar statement to \cite[Lemma~1.2]{Sch00}, but in the setting of counting autocorrelation instead of diffraction autocorrelation. }

\begin{lemma}\label{lem:gammacount_ud}
Let $\varLambda$ be a uniformly discrete point set in $\RR$, and let $\cA=(A_n)_n$ be a van Hove sequence of intervals for which $\card(F_n)=\card(\varLambda\cap A_n)\to \infty$ as $n\to\infty$.

Then for all $t \in \varLambda - \varLambda$ we have 
\[
\lim_{n \to \infty}\frac{\card(\varLambda\cap (-t+\varLambda)\cap A_n)- \card \left( F_n \cap (-t+F_n) \right)}{\card(F_n)} =0 \,.
\]
In particular, if $\varLambda$ has FLC, then 
$\eta_{\operatorname{count}}$ exists if and only if for all $t \in \RR$ the following limit exists 
\[
\eta_{\operatorname{count}}(t):= \lim_{n \to \infty} \frac{ \card \left( F_n \cap (-t+\varLambda) \right)}{ \card \left(F_n \right)}\, .
\]
Moreover, in this case, 
\[
\gamma_{\operatorname{count}}= \sum_{t \in \varLambda - \varLambda} \eta_{\operatorname{count}}(t) \delta_t \,.
\]
\end{lemma}

\begin{proof}
Let $t\in \RR$. Then the sets
\[
-t+(\varLambda\cap A_n) \quad\text{and}\quad
(-t+\varLambda)\cap A_n
\]
have symmetric difference contained in
\[
(-t+A_n)\;\triangle\; A_n.
\]
Because $A_n$ is an interval, this symmetric difference is no larger than two intervals of width $|t|$. Since $\varLambda$ is uniformly discrete, there is some minimum distance $r>0$ between points of $\varLambda$, and so $-t+(\varLambda\cap A_n)$ and $(-t+\varLambda)\cap A_n$ can differ by at most $2|t|/r$ points. Intersecting with $F_n=\varLambda\cap A_n$ shows that
\[
F_n\cap (-t+F_n) \quad\text{and}\quad
F_n\cap (-t+\varLambda) =
\varLambda\cap (-t+\varLambda)\cap A_n
\]
differ still by at most $2|t|/r$ elements.

By assumption, $\card(F_n)\to \infty$, and so $\displaystyle \lim_{n \to \infty} \frac{2|t|/r}{\card(F_n)}=0$, and therefore
\[
\lim_{n \to \infty}\frac{\card(F_n\cap (-t+F_n))-\card(\varLambda\cap(-t+\varLambda)\cap A_n)}{\card(F_n)} =0 \,.
\]
The remaining claims follow from  Lemma~\ref{lem:gammacount}. 
\end{proof}

\begin{remark} In $\RR^d$ the same proof fails, since the number of points in $\varLambda  \cap \left( (-t+A_n)\;\triangle\; A_n \right)$ can be unbounded as $n\to\infty$. To draw a similar conclusion for uniformly discrete point sets in $\RR^d$ using $(A_n)_n=([-n,n]^d)_n$ we would need the stronger assumption
\[
\lim_{n \to \infty} \frac{\card(F_n)}{n^{d-1}} = \infty \,.
\]
\end{remark}


Let us now cover a simple consequence of this, which we will use later in the paper.

\begin{proposition}\label{prop:no_autocorrelation_shift}
Let $\cA=(A_n)_n$ be a van Hove sequence consisting of intervals in $\RR$. If $\varLambda\subseteq \RR$ is a uniformly discrete point set with counting autocorrelation $\delta_0$ with respect to $\cA$, and for which $\displaystyle \lim_{n \to \infty} \card(\varLambda\cap A_n)=\infty$, then for any nonzero integer $k$, the set $\varLambda \cup (k+\varLambda)$ has counting autocorrelation and diffraction
\begin{align*}
\gamma_{\operatorname{count}} &=\delta_0 + \frac{\delta_{k}+\delta_{-k}}{2}\, , \\ 
\widehat{\gamma_{\operatorname{count}}}&=(1+\cos(2\pi k x)) \lambda \, ,
\end{align*}
with respect to $\cA$, where $\lambda$ is the Lebesgue measure.
\end{proposition}

\begin{proof}
Write $\varGamma= \varLambda \cup (k+\varLambda)$. By Lemma \ref{lem:gammacount_ud}, we have that $\gamma_{\operatorname{count}}=\sum_t \eta_{\operatorname{count}}(t)\delta_t$ where for $t\in \ZZ$,
\[
\eta_{\operatorname{count}}(t)=
\lim_{n\to \infty}\frac{\card(\varGamma\cap (-t+\varGamma)\cap A_n)}{\card(\varGamma\cap A_n)}.
\]
Let $\varLambda_n=\varLambda\cap A_n$ and similarly $\varGamma_n=\varGamma \cap A_n$, et cetera.  

The assumption that $\varLambda$ has counting autocorrelation $\delta_0$ means that whenever $s\in \RR$ is nonzero, $\card(\varLambda \cap (-s+\varLambda)\cap A_n)$ is $o(\card(\varLambda_n))$ as $n\to \infty$. As in the proof of Lemma \ref{lem:gammacount_ud}, because $\varLambda$ is uniformly discrete and $\card(\varLambda_n)\to \infty$, we have that $\card((-s+\varLambda)_n)\sim \card(-s+\varLambda_n)$ as $n\to\infty$, and similarly for $\varGamma$. Using inclusion-exclusion,
\begin{align*}
\card(\varGamma_n) &=
\card(\varLambda_n) + 
\card((k+\varLambda)_n)-
\card(\varLambda\cap (k+\varLambda)\cap A_n) \\ &\sim
\card(\varLambda_n)+\card(k+\varLambda_n)-
\card(\varLambda\cap (k+\varLambda)\cap A_n) \\ &=
2\card(\varLambda_n)-
\card(\varLambda\cap (k+\varLambda)\cap A_n), 
\end{align*}
and the last term is $o(\card(\varLambda_n))$, so $\card(\varGamma_n)\sim 2\card(\varLambda_n)$ as $n\to\infty$.

Now, let $t\in \RR$. The autocorrelation of $\varGamma$ considers the intersection
\begin{align*}
\varGamma\cap (-t+\varGamma) &=
(\varLambda \cap (-t+\varLambda))\;\cup\; (\varLambda \cap (k-t+\varLambda))\\ &\quad\cup\; 
((k+\varLambda)\cap (-t+\varLambda))\;\cup\; ((k+\varLambda)\cap (k-t+\varLambda)).
\end{align*}
Because $k \ne 0$, the autocorrelation of $\varLambda$ being $\delta_0$ implies that the intersection of any two or more sets in this union with $A_n$ has cardinality that is $o(\card(\varLambda_n))$ as $n\to \infty$. So, after using the inclusion-exclusion formula and discarding terms that go to zero, we find that
\begin{align*}
\eta_{\operatorname{count}}(t) &=
\lim_{n \to \infty}\frac{\card(\varGamma\cap (-t+\varGamma)\cap A_n)}{\card(\varGamma_n)} =
\lim_{n \to \infty}\frac{\card(\varGamma\cap (-t+\varGamma)\cap A_n)}{2\card(\varLambda_n)} \\ &=
\frac{1}{2}\lim_{n \to \infty}\left(
\frac{\card((\varLambda\cap (-t+\varLambda))_n)}{\card(\varLambda_n)}+\frac{\card((\varLambda\cap (k-t+\varLambda))_n)}{\card(\varLambda_n)}\right.\\ &\quad+\left.\frac{\card(((k+\varLambda)\cap (-t+\varLambda))_n)}{\card(\varLambda_n)} +\frac{\card(((k+\varLambda)\cap (k-t+\varLambda))_n)}{\card(\varLambda_n)}\right) \\ &=
\frac{1}{2}\left(\delta_0(t)+\delta_0(t-k)+\delta_0(t+k)+\delta_0(t)\right) \\ &=
\left(\delta_0+\frac{\delta_k+\delta_{-k}}{2}\right)(t).
\end{align*}
Therefore Lemma \ref{lem:gammacount_ud} implies that $\gamma_{\operatorname{count}}=\delta_0+(\delta_k+\delta_{-k})/2$ for $\varGamma$.
\end{proof}

We complete this subsection by showing that the addition of very few points does not change the counting autocorrelation.

It is well known that if two sets differ by a set of density zero, then the density autocorrelation of one exists if and only if the density autocorrelation of the other one exists. Moreover, in this case, the two autocorrelations are equal  {(see for example \cite[Corollary~6]{BHS} and its proof)}. We prove the corresponding result for counting autocorrelations.

\begin{lemma}\label{lemma:subset} Let $\varLambda \subseteq \varGamma  \subseteq G$ be uniformly discrete point sets and $\cA = (A_n)_n$ be a van Hove sequence with the following properties:
\begin{itemize}
\item[(a)] $\varLambda \cap A_n$ are eventually non-empty.
    \item[(b)] The autocorrelation $\gamma_{\operatorname{count}, \varLambda}$ exists with respect to $\cA$.
    \item[(c)] 
    \[
    \lim_{n \to \infty} \frac{\card(\varLambda \cap A_n)}{\card(\varGamma \cap A_n)} =1 \,.
    \]
\end{itemize}
Then, the autocorrelation $\gamma_{\operatorname{count}, \varGamma}$ exists with respect to $\cA$ and 
\[
\gamma_{\operatorname{count}, \varGamma}= \gamma_{\operatorname{count}, \varLambda} \,.
\]
\end{lemma}
\begin{proof} Let $F_n :=\varLambda \cap A_n $ and $E_n:= \varGamma \cap A_n$. We will show that 
\[
\frac{1}{\card(E_n)} \delta_{E_n}\ast \widetilde{\delta_{E_n}} -
\frac{1}{\card(F_n)} \delta_{F_n}\ast \widetilde{\delta_{F_n}} \to 0 
\]
vaguely, from which the result follows.

First let us note that 
\[
\frac{1}{\card(E_n)} \delta_{F_n}\ast \widetilde{\delta_{F_n}}-\frac{1}{\card(F_n)} \delta_{F_n}\ast \widetilde{\delta_{F_n}} = \left(\frac{\card(F_n)-\card(E_n)}{\card(E_n)} \right) \frac{1}{\card(F_n)}\left( \delta_{F_n}\ast \widetilde{\delta_{F_n}}\right)
\]
converges to $0 \cdot \gamma_{\operatorname{count}, \varLambda} =0$. Therefore, we need to show that 
\[
\frac{1}{\card(E_n)} \delta_{E_n}\ast \widetilde{\delta_{E_n}} -
\frac{1}{\card(E_n)} \delta_{F_n}\ast \widetilde{\delta_{F_n}} \to 0.
\]

Let $f \in \Cc(G)$ be arbitrary. Then
\begin{align*}
&\left| \frac{1}{\card(E_n)} \delta_{E_n}\ast \widetilde{\delta_{E_n}} (f) -
\frac{1}{\card(E_n)} \delta_{F_n}\ast \widetilde{\delta_{F_n}} (f) \right|  \\
&\leq \left| \frac{1}{\card(E_n)} \delta_{E_n}\ast \widetilde{\delta_{E_n}} (f) -
\frac{1}{\card(E_n)} \delta_{F_n}\ast \widetilde{\delta_{E_n}} (f) \right|  \\
&+
\left| \frac{1}{\card(E_n)} \delta_{F_n}\ast \widetilde{\delta_{E_n}} (f) -
\frac{1}{\card(E_n)} \delta_{F_n}\ast \widetilde{\delta_{F_n}} (f) \right|  \\
&\leq  \frac{1}{\card(E_n)} \left|\left( \delta_{E_n} -\delta_{F_n} \right) \ast \widetilde{\delta_{E_n}} (f) \right| +
\frac{1}{\card(E_n)} \left|  \delta_{F_n}\ast \left( \widetilde{\delta_{E_n}-\delta_{F_n}}\right) (f) \right|  \\
&= \frac{1}{\card(E_n)} \left|\int_{G}  \int_{G} f(x+y) \dd \widetilde{\delta_{E_n}}(x) \dd \left( \delta_{E_n} -\delta_{F_n} \right)(y) \right| \\
&+ 
\frac{1}{\card(E_n)} \left|  \int_{G}  \int_{G} f(x+y) \dd \delta_{F_n}(x) \dd \left( \widetilde{ \delta_{E_n} -\delta_{F_n}} \right)(y)\right|  \\
&\leq \frac{1}{\card(E_n)} \int_{G}  \int_{G} \left|f(x+y)\right| \dd \widetilde{\delta_{E_n}}(x) \dd \left( \delta_{E_n\backslash F_n} \right)(y)  \\
&+ 
\frac{1}{\card(E_n)}   \int_{G}  \int_{G} \left|f(x+y)\right| \dd \delta_{F_n}(x) \dd \left( \widetilde{ \delta_{E_n \backslash F_n}} \right)(y)\\
&\leq \frac{1}{\card(E_n)} \int_{G}  \int_{G} \left|f(x+y)\right| \dd \widetilde{\delta_{\varGamma}}(x) \dd \left( \delta_{E_n\backslash F_n} \right)(y)  \\&+ 
\frac{1}{\card(E_n)}   \int_{G}  \int_{G} \left|f(x+y)\right| \dd \delta_{\varLambda}(x) \dd \left( \widetilde{ \delta_{E_n \backslash F_n}} \right)(y) \,.
\end{align*}
Now, since $f \in \Cc(G)$ and $\varGamma$ is uniformly discrete, all measures involved are translation bounded and so there exists a constant $C$ such that for all $y \in G$ we have 
\begin{align*} 
 \int_{G} \left|f(x+y)\right| \dd \delta_{\varLambda}(x) &  \leq C \quad\text{and} \\
  \int_{G} \left|f(x+y)\right| \dd \widetilde{\delta_{\varGamma}}(x) &\leq C. 
\end{align*}
It follows that 
\begin{align*}
 &\left| \frac{1}{\card(E_n)} \delta_{E_n}\ast \widetilde{\delta_{E_n}} (f) -
\frac{1}{\card(E_n)} \delta_{F_n}\ast \widetilde{\delta_{F_n}} (f) \right|  \\
&\leq \frac{1}{\card(E_n)} \int_{G}  C\: \dd \left( \delta_{E_n\backslash F_n} \right)(y)  + 
\frac{1}{\card(E_n)}   \int_{G}  C\: \dd \left( \widetilde{ \delta_{E_n \backslash F_n}} \right)(y) \\
&=\frac{C}{\card(E_n)} 2 \card(E_n \backslash F_n) \to 0 \,,
\end{align*}
completing the proof.

\end{proof}

\subsection{Diffraction of subsets of integers}

As mentioned before, many examples we will consider below are subsets $\varLambda$ of the integers, considered as point sets in $G=\RR$. The averaging van Hove sequences $(A_n)_n$ will be of the form 
$
A_n=[a_n,b_n] \,,
$
where $\displaystyle \lim_{n \to \infty} (b_n-a_n) =\infty$. 

\smallskip

Let us now fix $\varLambda  \subseteq \ZZ$ and $(A_n)_n$ a van Hove sequence in $\RR$.
For each $n$, let $f_n, g_n :\ZZ \to \CC$ be defined as
\begin{align*}
f_n(d)&:= \frac{1}{\vol(A_n)} \card (  \varLambda \cap (d+\varLambda) \cap A_n) \\
g_n(d)&:= \frac{ \card (  \varLambda \cap (d+\varLambda) \cap A_n)}{ \card (  \varLambda\cap A_n)} = \frac{f_n(d)}{f_n(0)} \,.
\end{align*}
We then have {(compare \cite[Sec.~10.3.2]{TAO} or \cite[Thm.~1]{Ba})}:

\begin{proposition} Let $\varLambda \subseteq \ZZ$ and let $(A_n)_n$ be a van Hove sequence in $\RR$. 
Then, $\gamma_{\operatorname{dens}}$ exists with respect to $(A_n)_n$ if and only if $f_n$ converges pointwise to a function $f : \ZZ \to \CC$.

Moreover, in this case $\gamma_{\operatorname{dens}}= \sum_{m \in \ZZ} f(m) \delta_m$ and $f$ is positive definite. Finally, if $\sigma$ is the positive measure on $\RR/\ZZ$ corresponding to $f$ via Bochner's theorem, then $\sigma$ and the 1-periodic measure $\widehat{\gamma_{\operatorname{dens}}}$  are related via the so-called Weil formula: 
\[
\int_{\RR} f(x) \mathrm d \widehat{\gamma_{\operatorname{dens}}}(x) = \int_{\RR/\ZZ} \mathcal{W}f(x+\ZZ) \mathrm d \sigma(x) \qquad  \forall f \in \Cc(G) \,,
\]
where 
\[
\mathcal{W}f(x+\ZZ) := \sum_{y \in \ZZ} f(x+y)  \qquad  \forall x \in \RR \,.
\]
\end{proposition}
\begin{proof}
The first part follows trivially from Proposition~\ref{lem:gammadens}. The last part follows from \cite{RS}.
\end{proof}

Let us note here in passing that intuitively, the Weil formula says that $\widehat{\gamma_{\operatorname{dens}}}$ is the $1$-periodic measure obtained by ``periodizing" $\sigma$. More precisely, $\sigma$ is the measure obtained by restricting $\widehat{\gamma_{\operatorname{dens}}}$ to $[0,1)$ and identifying this interval with $\RR/\ZZ$.

\medskip

Now, exactly the same result holds for the counting autocorrelation. Since the proof is similar, we skip it.

\begin{proposition} Let $\varLambda \subseteq \ZZ$ and let $(A_n)_n$ be a van Hove sequence in $\RR$. 
Then, $\gamma_{\operatorname{count}}$ exists with respect to $(A_n)_n$ if and only if $g_n$ converges pointwise to a function $g : \ZZ \to \CC$.

Moreover, in this case $\gamma_{\operatorname{count}}= \sum_{m \in \ZZ} g(m) \delta_m$ and $g$ is positive definite. Finally, if $\sigma$ is the positive measure on $\RR/\ZZ$ corresponding to $g$ via Bochner's theorem, then $\sigma$ and the 1-periodic measure $\widehat{\gamma_{\operatorname{count}}}$  are related via the Weil formula.\qed
\end{proposition}

\section{Diffraction of sets with no infinite translated subsets}\label{sect:boring}

As we saw in the last section, the autocorrelation encodes how often an element $t \in G$ appears as the difference of two points of $\varLambda$. The easiest examples to study, then, are those sets with periodic subsets or those where $\varLambda$ and $t+\varLambda$ have little agreement. In this section we look at the former situation.

\begin{lemma}\label{lem:finitetranslates}
Let $\varLambda \subseteq G$ be a set of FLC and let $\cA=(A_n)_n$ be a van Hove sequence.
Assume that for $F_n = \varLambda \cap A_n$
\begin{itemize}
    \item[(a)] $\displaystyle \lim_{n \to \infty} \card(F_n)=\infty$.
    \item[(b)] For all $0 \neq t \in \varLambda -\varLambda$ we have 
    \[
    \card(\varLambda \cap (-t+\varLambda)) < \infty \,.
    \]
\end{itemize}
Then, the counting autocorrelation is $\gamma_{\operatorname{count}}=\delta_0$ and it exists with respect to $\cA$.
\end{lemma}

\begin{proof}
For any nonzere $0\ne t\in \varLambda-\varLambda$, we can bound
\[
\frac{\card(F_n\cap (-t+F_n))}{\card(F_n)} \le
\frac{\card(\varLambda \cap (-t+\varLambda))}{\card(F_n)},
\]
where assumptions (a) and (b) imply that the right-hand side tends to zero. So, for nonzero $t$, in the notation of Lemma \ref{lem:gammacount}, we have
\[
\eta_{\operatorname{count}}(t) = \lim_{n \to \infty} \frac{\card(F_n\cap (-t+F_n))}{\card(F_n)} =0.
\]
For $t=0$, $\eta_{\operatorname{count}}(0)=1$, because $F_n$ is eventually nonempty. So, each limit $\eta_{\operatorname{count}}(t)$ exists, and so by Lemma \ref{lem:gammacount}, the autocorrelation $\gamma_{\operatorname{count}}$ exists with respect to $\cA$ and it is
\[
\gamma_{\operatorname{count}}=
\sum_{t\in \varLambda-\varLambda} \eta_{\operatorname{count}}(t)\delta_t =
\delta_0.\qedhere
\]
\end{proof}



\begin{remark} \
\begin{itemize}
    \item[(a)]
Let $\varLambda \subseteq G$ be locally finite and let $K \subseteq G$ be compact. A short computation shows that 
\begin{align*}
 \quad \quad \ \ \ \left|\gamma^{}_{F^{}_n}-\delta_0 \right|(K) &=  \left|\gamma^{}_{F^{}_n}-\delta_0 \right|(\{0\}) +  \left|\gamma^{}_{F^{}_n}\right|(K \backslash \{0 \})\\
&= \left| \gamma^{}_{F^{}_n}(\{0\})-1 \right| +\frac{1}{\card(F_n)} \sum_{x,y \in F_n} \delta_{x-y}(K \backslash \{0 \}) \\
&=\left| \frac{1}{\card(F_n)} \delta_{F_n}\ast\widetilde{\delta_{F_n}}(\{0\}) -1 \right|
+\frac{1}{\card(F_n)}\sum_{\substack{t \in (F_n -F_n) \cap K \\ t \neq 0}} \card(F_n \cap (-t+F_n))
\\
&\leq  \left| \frac{1}{\card(F_n)} \sum_{x,y \in F_n} \delta_{x-y} (\{0\}) \right| +  \frac{1}{\card(F_n)}   \sum_{\substack{t \in (\varLambda -\varLambda) \cap K \\ t \neq 0}} \card(\varLambda \cap (-t+\varLambda)).
\end{align*}
Now, if $\varLambda$ has FLC, then the set $(\varLambda -\varLambda) \cap K$ is finite. Therefore, if the assumptions in Lemma~\ref{lem:finitetranslates} hold, then we get the stronger conclusion that 
\[
\left| \gamma^{}_{F^{}_n} -\delta_0 \right|(K) \to 0
\]
for all compact sets $K$.
\item[(b)] The conclusion of Lemma~\ref{lem:finitetranslates} does not hold for sets without FLC. For example, 
\[
\varLambda = \bigcup_n \left\{ n!+\frac{1}{e^n}, \ldots, n!+  \frac{\lfloor e^n \rfloor}{e^n} \right\}
\]
has density zero, meets any of its translates in at most one point, but one can show that
\[
\gamma_{\operatorname{count}}= \lambda|_{[-1,1]} \,.
\]
\end{itemize}
\end{remark}



As an immediate consequence of Lemma~\ref{lem:finitetranslates} we get:

\begin{lemma} Let $\varLambda =\{ a_n : n\in \NN \} \subseteq \RR$ for any sequence $(a_n)_n$ satisfying 
\[
\lim_{n \to \infty} \;(a_{n+1}-a_n) =\infty
\]
Then, $\varLambda$ has FLC, and for any van Hove sequence $(A_n)_n$ such that $\displaystyle \lim_{n \to \infty} (A_n \cap \varLambda)= \infty$ we have  
\begin{equation}\label{eq23}
\gamma_{\operatorname{count}}=\delta_0  \, ,\quad  \reallywidehat{\gamma_{\operatorname{count}}}=\lambda \,.
\end{equation}
In particular, \eqref{eq23} holds for all van Hove sequences which are nested and exhausting.
\end{lemma}

\begin{proof}
First note that for each $R>0$ there exists some $N$ such that, for all $n >N$ we have 
\[
\varLambda \cap [a_n-R, a_n+R] = \{ a_n \} \,.
\]
It follows immediatelly that $\varLambda$ has FLC.

Note that the given condition implies that for all $t \neq 0$ the set
\[
\varLambda \cap (t+\varLambda)  
\]
is finite. Therefore, for any van Hove sequence for which $\displaystyle \lim_{n \to \infty} \card(F_n)=\infty$, we get $\gamma_{\operatorname{count}}=\delta_0$ by Lemma \ref{lem:finitetranslates}.

Finally, if $(A_n)_n$ is exhausting, $\bigcup_n A_n=\RR$ and so $\bigcup_n F_n=\varLambda$. If $(A_n)_n$ is nested, so is the sequence $(F_n)_n$, and so since $\varLambda$ is infinite, we must have $\displaystyle \lim_{n \to \infty} \card(F_n)=\infty$.
\end{proof}

The following two-sided version is proven in exactly the same way, and we skip the details. 

\begin{lemma} Let $\varLambda =\{ a_n : n\in \ZZ \} \subseteq \RR$ so that 
\begin{align*}
\lim_{n \to \infty} a_{n+1}-a_n &=\infty \quad\text{and} \\
\lim_{n \to -\infty} a_{n-1}-a_{n} &=-\infty.
\end{align*}
Then, $\varLambda$ has FLC, and for any van Hove sequence $(A_n)_n$ such that $\displaystyle \lim_{n \to \infty} (A_n \cap \varLambda)= \infty$ we have  
\begin{equation}\label{eq24}
\gamma_{\operatorname{count}}=\delta_0  \, ,\quad  \reallywidehat{\gamma_{\operatorname{count}}}=\lambda \,.
\end{equation}
In particular, \eqref{eq24} holds for all van Hove sequences which are nested and exhausting.
\end{lemma}

The next three examples follow immediately from these results. 
\begin{example} Choose any $a >1$, let 
\[
\varLambda = \{a^n : n \in \NN \},
\]
and let $(A_n)_n=([-n,n])_n$. Then $\gamma_{\operatorname{count}}=\delta_0$ and $\widehat{\gamma_{\operatorname{count}}}=\lambda$. \qed 
\end{example}

\begin{example} Let 
\[
\varLambda = \{n! : n \in \NN \}
\]
and let $(A_n)_n=([-n,n])_n$. Then $\gamma_{\operatorname{count}}=\delta_0$ and $\widehat{\gamma_{\operatorname{count}}}=\lambda$. \qed 
\end{example}

\begin{example} Let 
\[
\varLambda = \{f_n : n \in \NN \}
\]
where $(f_n)_n$ is the Fibonacci sequence $f_0=0, f_1=1$ and $f_{n+1}=f_n+f_{n-1}$. Let $(A_n)_n=([-n,n])_n$.
Then $\gamma_{\operatorname{count}}=\delta_0$ and $\widehat{\gamma_{\operatorname{count}}}=\lambda$. \qed
\end{example}

More generally, the same behaviour is exhibited by many point sets arising from linear recurrence relations, in particular when the largest root in modulus of the characteristic polynomial is strictly greater than 1. The point being that all of these have approximately exponential growth.

It should be noted that these examples can be put into Proposition~\ref{prop:no_autocorrelation_shift} to get a point set with diffraction and autocorrelation 
$$\gamma_{\operatorname{count}} =\delta_0 + \frac{\delta_{k}+\delta_{-k}}{2} \quad  \textrm{and} \quad 
\widehat{\gamma_{\operatorname{count}}}=(1+\cos(2\pi k x)) \lambda\, . $$

\section{Diffraction of the primes}\label{sect:primes}

In this section, we discuss the diffraction of primes $\PP \subseteq \ZZ$ (including negative primes) for van Hove sequences of intervals. We will see that the density autocorrelation is always $0$, while the counting autocorrelation is, under weak conditions, $\delta_0$.

Throughout this entire section, for $d \in \NN$, let us recall the notation \begin{align*}
\pi_d(x)= \card \{ p \in \PP \cap [0,x] : p+d \in \PP \} \,.
\end{align*}
Of course, $\pi_0(x)$ is just the prime counting function $\pi(x)$. Since there is only one even prime, for odd $d$ we have 
\[
\pi_d(x) \leq 1
\]
for any $x$.


\subsection{Density autocorrelation with respect to van Hove sequences of intervals.}

Consider a van Hove sequence $(A_n)_n=([a_n,b_n])_n$ of intervals. By eventually replacing $a_n,b_n$ by nearby integers, we can assume without loss of generality that $a_n,b_n \in \ZZ$.

\medskip

A simple consequence of the Brun-Titchmarsh Theorem is that the density of the primes with respect to $(A_n)_n$ is zero. Indeed, let us recall the following version of this theorem:

\begin{theorem}\cite[Corollary~2]{MV}\label{thm:mv} Let $m,n$ be positive integers. Then
\[
\pi(m+n)-\pi(m) \leq 2 \pi(n) \,.
\]\qed
\end{theorem}

As an immediate consequence, we get:

\begin{theorem}\label{thm:diff-pri} Let $(A_n)_n=([a_n,b_n])_n$ be any van Hove sequence of intervals. 
Then,
\begin{align*}
\card(\PP \cap A_n) \leq \frac{2 \pi(b_n-a_n)}{b_n-a_n} \,.
\end{align*}
In particular, with respect to $(A_n)_n$ we have $\dens(\PP)=0$ and so the density diffraction is
\[
\gamma_{\operatorname{dens}}=0 \,.
\]
\end{theorem}
\begin{proof}
Let $r_n=b_n-a_n$. Then, since $(A_n)_n$ is a van Hove sequence, we have $\displaystyle \lim_{n \to \infty} r_n = \infty$.

Let $F_n:= \PP \cap A_n$.

We split the problem into three cases:

\underline{Case 1:} Suppose $0<a_n$. Then, by Theorem~\ref{thm:mv} we have
\begin{align*}
\frac{\card(F_n)}{r_n}= \frac{\pi(b_n)-\pi(a_n)}{r_n} \leq 2 \frac{ \pi(r_n)}{r_n} \,.
\end{align*}

\underline{Case 2:} Suppose $b_n<0$. Then, by Theorem~\ref{thm:mv} we have
\begin{align*}
\frac{\card(F_n)}{r_n}= \frac{\pi(|a_n|)-\pi(|b_n|)}{r_n} \leq 2 \frac{ \pi(r_n)}{r_n} \,.
\end{align*}

\underline{Case 3:} Suppose $a_n \leq 0 \leq b_n$. Then, $|a_n|<r_n$ and $b_n <r_n$, and hence
\begin{align*}
\frac{\card(F_n)}{r_n}= \frac{\pi(b_n)+\pi(|a_n|)}{r_n} \leq 2 \frac{ \pi(r_n)}{r_n} \,.
\end{align*}

Therefore, for all $n$ we have 
\[
\frac{\card(F_n)}{r_n} \leq  2 \frac{ \pi(r_n)}{r_n} \,.
\]
Since $r_n \to \infty$, the last claim follows from the Prime Number Theorem.
\end{proof}

\begin{remark} Let $\XX(\PP)$ be the dynamical system generated by $\PP$ under the translation action of $\RR$, see \cite{BL} for details. Then, { since the primes have arbitrarily large gaps}, we have $\emptyset \in \XX(\PP)$.

Now, for each $\varphi_1, \ldots, \varphi_n \in \Cc(\RR)$ define $F_{\varphi_1, \ldots, \varphi_n} : \XX(\PP) \to \CC$ via
\[
F_{\varphi_1, \ldots, \varphi_n}(\varLambda) := \prod_{j=1}^n \left(\sum_{x \in \varLambda} \varphi_j(x) \right) \,.
\]
Let {
\[
\AAA:= \{ F_{\varphi_1, \ldots, \varphi_n} +c 1_{\XX(\PP)} : 0 \leq n,\varphi_1, \ldots, \varphi_n \in \Cc(\RR) , c \in \RR \} \,. 
\]}
Then, $\AAA \subseteq C(\XX(\PP))$ is a subalgebra separating the points and hence dense in $C(\XX(\PP))$ by Stone--Weierstrass.

Now, it follows immediately from Theorem~\ref{thm:diff-pri} that 
{
\[
\lim_{k \to \infty} \frac{1}{2k} \int_{t-k}^{t+k} F_{\varphi_1, \ldots, \varphi_n}(s+\PP) \dd s =0  = \int_{\XX(\PP)} F_{\varphi_1, \ldots, \varphi_n}(\varGamma) \dd \delta_{\emptyset}(\varGamma) \,. 
\]
uniformly in $t$. Also, we trivially have 
\[
\lim_{k \to \infty} \frac{1}{2k} \int_{t-k}^{t+k} c 1_{\XX(\PP)}(s+\PP) \dd s =c  = \int_{\XX(\PP)} c 1_{\XX(\PP)}(\varGamma) \dd \delta_{\emptyset}(\varGamma) \,. 
\]
}
The density of $\AAA$ in  $C(\XX(\PP))$ immediately implies that $(\XX(\PP), \RR)$ is uniquely ergodic, with unique ergodic measure {$m=\delta_\emptyset$}. Therefore, by \cite{BL}, the primes must have density and density autocorrelation zero with respect to any van Hove sequence (not necessarily of intervals).

The same conclusion holds if we replace the translation action of $\RR$ by the translation action of $\ZZ$.
 
\end{remark}

\subsection{Counting diffraction with respect to van Hove sequences of intervals}

Numerical estimates have reported that the diffraction of primes is pure point, meaning that the primes are a model for quasicrystals \cite{TZC2} (compare \cite{TZC1,ZMT}). This would mean that the (counting) diffraction would have the form
\[
\reallywidehat{\gamma_{\operatorname{count}}} =\sum_{y \in B} I(y) \delta_y \,.
\]
This seems too nice to be true. Indeed, for a set $\varLambda \subseteq \ZZ$, the much weaker statement that $\reallywidehat{\gamma_{\operatorname{count}}} \neq \lambda$ is equivalent to the fact that there exists some $d \in \ZZ \backslash \{ 0 \}$ such that 
\[
\limsup_{n \to \infty} \frac{\card(\varLambda \cap (d+\varLambda) \cap [-n,n])}{\card(\varLambda \cap [-n,n])} >0 \,.
\]
When $\varLambda = \PP$, this would imply that there exists a positive integer $d>0$ such that the set $\{ p \in \PP: p+d \in \PP \}$ would have positive density among the primes.

As mentioned above, such a result is too nice to be true. It would give some results which are much much stronger than the latest developments in number theory (\cite{Zhang,Polymath,Maynard}).
In fact, it turns out that this would contradict {the generalisation of the Brun-Titchmarsh Theorem, Proposition~\ref{prop:dens_d-primes}}, and hence cannot be true. This means that the counting diffraction of the primes must be the Lebesgue measure
\[
\reallywidehat{\gamma_{\operatorname{count}}} = \lambda \,.
\]

\smallskip
We now prove that this is indeed the case. We show that the counting autocorrelation $\gamma_{\operatorname{count}}$ exists along many van Hove sequences of intervals and it is $
\gamma_{\operatorname{count}}= \delta_0$. This implies that the diffraction is the Lebesgue measure and hence absolutely continuous.

\medskip

{ First, the aforementioned generalisation of the Brun-Titchmarsh Theorem which is itself a generalisation of Brun's Theorem:}

\begin{proposition}\label{prop:dens_d-primes}
For all $d \geq 1$ we have 
\[
\lim_{x \to \infty} \frac{\pi_d(x)}{\pi_0(x)}=0 \,.
\]
\end{proposition}
\begin{proof}
For even $d$, it is well known using sieve methods in number theory that there exists a constant $C>0$ such that
\[
\pi_d(x) \leq C\frac{x}{\log^2 x}, \quad \forall x\geq 1.
\]
Since we are not experts, we do not know precisely who came up with this argument but it can be found in Bateman and Stemmler \cite[Lemma 3]{BateStem}.
Therefore, by the Prime Number Theorem,
\[
0 \leq \limsup_{x \to \infty} \frac{\pi_d(x)}{\pi_0(x)}  
\leq \lim_{x\to\infty} \frac{C x/\log^2 x}{x/\log x} = 0,
\]
so $\displaystyle \lim_{x\to \infty} \pi_d(x)/\pi_0(x)=0$.
\end{proof}

Now note  that if $F \subseteq \RR$ is any finite set we have \[
\gamma^{}_{-F}=\gamma^{}_{F} \,.
\]
Since $-\PP=\PP$, for any set $A$ we have 
\begin{equation}\label{eq:sym}
-(\PP \cap A) =\PP \cap (-A) \,.
\end{equation}
Therefore, we immediately get :

\begin{lemma}\label{lem:flip} Let $\cA = (A_n)_n$ be any van Hove sequence and with $\epsilon_n \in \{ \pm 1 \}$ let $\cB = (B_n)_n=(\epsilon_n A_n)_n$. Then, the counting autocorrelation $\gamma_{\operatorname{count},\cA}$ of $\PP$ exists with respect to $\cA$ if and only if the counting autocorrelation $\gamma_{\operatorname{count},\cB}$ of $\PP$ exists with respect to $\cB$.
Moreover, in this case $\gamma_{\operatorname{count},\cA} =\gamma_{\operatorname{count},\cB}$.\qed
\end{lemma}

Whenever when we are given a van Hove sequence $(A_n)_n=([a_n,b_n])_n$ of intervals, Lemma~\ref{lem:flip} will allow us to flip some of those intervals so that $|b_n|\geq |a_n|$. Since, after the flip, we also have $a_n < b_n$, we get that $b_n \geq 0$ and must have
\[
\lim_{n \to \infty} b_n =\infty \, ,
\]
because the width of the intervals goes to infinity. We will often make this extra assumption.

Next, let us cover the following two results.

\begin{lemma}\label{lem:0_in_van_Hove} Let $(A_n)_n=([a_n,b_n])_n$ be a van Hove sequence of intervals such that for all $n$ we have $0 \in A_n$. Then, with respect to $(A_n)_n$ we have $\gamma_{\operatorname{count}}=\delta_0$. 
\end{lemma}
\begin{proof}
By \eqref{eq:sym} and Lemma \ref{lem:flip} we can assume without loss of generality that $|b_n| \geq |a_n|$. This implies that $b_n >0$ and $\displaystyle \lim_{n \to \infty} b_n =\infty$. 
As usual, let $F_n=\PP \cap A_n$.

For each $t \neq 0$ we have 
\begin{align*}
\card(F_n \cap(t+\PP))& \leq 2 \pi_{|t|}(b_n)\, , \\
\card(F_n) &\geq  \pi(b_n) \,.
\end{align*}
The claim now follows from Proposition~\ref{prop:dens_d-primes} and Lemma~\ref{lem:gammacount_ud}.
\end{proof}

As a consequence we get:

\begin{theorem}\label{thm:primes_counting_diffraction} The primes $\PP$ have counting autocorrelation and diffraction with respect to the van Hove sequence $(A_n)_n =([-n,n])_n$
\begin{align*}
\gamma_{\operatorname{count}}&= \delta_0, \\
\widehat{\gamma_{\operatorname{count}}}&= \lambda  \,.
\end{align*}\qed
\end{theorem}

\begin{remark}\label{rem:diff-primes-conv} Combining the proofs of Lemma~\ref{lem:0_in_van_Hove} , Proposition~\ref{prop:dens_d-primes}, Lemma~\ref{lem:gammacount_ud} and the Prime Number Theorem, it follows that, for each van Hove sequence $(A_n)_n$ of intervals such that $0 \in A_n$ there exists a constant $C_1>0$ such that for all $m \in \ZZ$ we have 
\[
\left| \gamma^{}_{F^{}_n} - \delta_0 \right| (m) \leq  \frac{2\pi_{|m|}(n)+|m|}{\pi_0(n)} \leq C_1 \left( \frac{1}{\log(n) } + \frac{ |m| \log(n)}{n} \right) \,.
\]
This seem to suggest that $\gamma^{}_{F^{}_n}$ converges very slowly to $\gamma_{\operatorname{count}}$, meaning that numerical estimates are useless when studying the diffraction of primes.

{
Note also here that for the density autocorrelation calculated with respect to $(A_n)_n=([-n,n])_n$ we have 
\begin{align*}
\left| \gamma^{}_{n} - 0 \right| (m) &\leq  \frac{\pi(n)}{n} \left| \gamma^{}_{F^{}_n} - \delta_0 \right| (m)+ \frac{\pi(n)}{n}\delta_0 (m)  \\
&\leq \frac{C_2}{\log(n)}C_1 \left( \frac{1}{\log(n) } + \frac{ |m| \log(n)}{n} \right)+ \frac{C_2}{\log(n)} \\
&= C \left( \frac{1}{\log^2(n) } + \frac{ |m|}{n} \right)+ \frac{C_2}{\log(n)} \,,
\end{align*}
for some constants $C$ and $C_2$. This explains the slow rate of convergence of the density autocorrelation to $0$.

A similar estimate can be obtained with respect to any van Hove sequence $(A_n)_n$ with $0 \in A_n$ for all $n\geq 1$.

}
\end{remark}

\begin{lemma}\label{lem5.7}
Let $(A_n)_n=([a_n,b_n])_n$ be a van Hove sequence of intervals such that for all $n$ we have $0 < a_n <b_n$. Assume that there exists some $c>1$ and $N$ such that for all $n>N$ we have $b_n \geq ca_n$. 

Then, with respect to $A_n$ the primes have counting autocorrelation $\gamma_{\operatorname{count}}=\delta_0$. 
\end{lemma}
\begin{proof}
For each $t \neq 0$ we have 
\begin{align*}
\card(F_n \cap(t+\PP))& \leq  \pi_{|t|}(b_n)\, , \\
\card(F_n) &=  \pi(b_n) -\pi(a_n) \geq \pi(b_n)-\pi(b_n/c)\, .
\end{align*}
Now, by Proposition~\ref{prop:dens_d-primes} and the Prime Number Theorem we have 
\begin{align*}
\lim_{n \to \infty} \frac{\pi_{|t|}(b_n) }{\frac{b_n}{\ln(b_n)}}&=0 \quad\text{and} \\
\lim_{n \to \infty} \frac{\pi(b_n)-\pi(b_n/c) }{\frac{b_n}{\ln(b_n)}}&=1-\frac{1}{c}>0 \,.
\end{align*}
It follows that for all $t \neq 0$ we have 
\[
\lim_{n \to \infty} \frac{\card(F_n \cap(t+\PP))}{\card(F_n)} =
\lim_{n \to \infty} \frac{\pi_{|t|}(b_n)}{\pi(b_n)-\pi(b_n/c)}=0 \,.
\]
The claim follows from Lemma~\ref{lem:gammacount_ud}.
\end{proof}

\begin{remark}\label{rem:primes_can_have_0_diffraction}
The condition that $b_n \geq ca_n$ is important and cannot be dropped.  For instance, $(A_n)_n = ([n!, n!+n])_n$ is a van Hove sequence that intersects $\PP$ in at most one point, because $n!+k$ is always divisible by $k$. With respect to this van Hove sequence we have 
\[
\gamma_{\operatorname{count}}=0 \,.
\]
\end{remark}

Let us now state a few immediate consequences of the previous result.

\begin{corollary}\label{chem-are-wrong} Let $a_n,b_n$ be any sequences of integers such that $b_n>a_n>0$ and $\displaystyle \lim_{n \to \infty} b_n = \infty$. If
\[
\lim_{n \to \infty} \frac{b_n}{a_n}=L > 1 \,.
\]
Then, $(A_n)_n=([a_n,b_n])_n$ is a van Hove sequence and with respect to this van Hove sequence we have
\[
\gamma_{\operatorname{count}}= \delta_0 \,, \quad\, \reallywidehat{\gamma_{\operatorname{count}}}=\lambda \,.
\]\qed
\end{corollary}

\begin{remark} In \cite{ZMT, TZC1, TZC2} the authors are considering intervals of the form $[M_n, M_n+L_n]$ where $M_n \to \infty$ and the ratio $\frac{L_n}{M_n}$ converges to some $\beta >0$.
The setting falls within the assumptions from Corollary~\ref{chem-are-wrong}, and 
\[
L=1+\beta >1 \,.
\]
Therefore, with respect to such van Hove sequence of intervals, the diffraction is actually the Lebesgue measure $\lambda$, and hence absolutely continuous.
\end{remark}

We can now prove our most general version of these results.

\begin{theorem}\label{thm:main1} Let $(A_n)_n=([a_n,b_n])_n$ be any van Hove sequence of intervals. Assume that there exists some constant $d>0$ such that 
\[
|a_n| < d(b_n-a_n) \,.
\]
Then, with respect to $(A_n)_n$ we have $\gamma_{\operatorname{count}}=\delta_0$ and $\reallywidehat{\gamma_{\operatorname{count}}}=\lambda$. 
\end{theorem}
\begin{proof}
As usual, let $r_n=b_n-a_n$ .

First let us note that 
\[
|b_n| \leq |a_n|+|b_n-a_n| \leq (d+1)r_n \,.
\]
In particular, for all $n$ we have 
\[
\max \{ |a_n|, |b_n| \} \leq (d+1) r_n \,.
\]
Define 
\[
B_n= \left\{
\begin{array}{cc}
 \bigl[ a_n,b_n \bigr] ,   & \mbox{ if } |a_n| \leq |b_n| \,, \\
  \bigl[ -b_n, -a_n \bigr],   & \mbox{ if } |a_n| \geq |b_n| \,.
\end{array}
\right.
\]
Then, by Lemma~\ref{lem:flip}, it suffices to prove that our conclusion holds with respect to $(B_n)_n$. Note that if $(B_n)_n=([a_n', b_n'])_n$ we have 
\[
|b_n'| \leq (d+1)(b_n'-a_n') \,.
\]

Therefore, without loss of generality, we can assume that $|b_n| \geq |a_n|$ for all $n$.
As before, this implies that $b_n >0$ and $b_n\to\infty$. 

Define 
\begin{align*}
A&:= \{ n : a_n \leq 0 \} \\
B&:= \{ n : a_n >0 \}
\end{align*}

We split the proof into three cases.

\underline{Case 1:} $A$ is finite. Then, by eventually erasing the first few terms of $A_n$, we can assume that $a_n >0$ for all $n$.

Next, we have 
\[
\frac{b_n}{a_n} = \frac{b_n-a_n}{a_n}+1 \geq d+1 >1 \,.
\]
The conclusion follows from Lemma~\ref{lem5.7}.

\underline{Case 2:} $B$ is finite. Then, by eventually erasing the first few terms of $A_n$, we can assume that $a_n \leq 0$ for all $n$. 
The conclusion then follows from Lemma~\ref{lem:0_in_van_Hove}.

\underline{Case 3:} $A, B$ are infinite. Then write $A$ as the sequence $k_1, k_2, \ldots, k_n \ldots$ and $B$ as the sequence $m_1, m_2, \ldots, m_n \ldots$.

By Lemma~\ref{lem:0_in_van_Hove}, as in Case 2, we have 
\begin{align*}
\lim_{n \to \infty} \gamma^{}_{F_{k_n}} &=\delta_0 \,.
\end{align*}
Also, exactly as in Case 1, we have 
\[
\frac{b_{m_n}}{a_{m_n}} \geq d+1 >1 \,.
\]
Therefore, by Lemma~\ref{lem5.7} we have 
\begin{align*}
\lim_{n \to \infty} \gamma^{}_{F_{m_n}} &=\delta_0 \,.
\end{align*}
Since $A \cup B =\NN$ we conclude that
\begin{align*}
\lim_{n \to \infty} \gamma^{}_{F^{}_{n}} &=\delta_0 \,,
\end{align*}
as claimed.
\end{proof}

\begin{corollary} Let $(A_n)_n$ be a van Hove sequence of intervals which have a common point $c$. Then, with respect to $(A_n)_n$ we have $\gamma_{\operatorname{count}}=\delta_0$ and $\reallywidehat{\gamma_{\operatorname{count}}}=\lambda$.

In particular, the conclusion holds for any nested van Hove sequence of intervals. 
\end{corollary}
\begin{proof} For each $n\geq 1$, let $A_n=[a_n,b_n]$ and let $r_n=b_n-a_n$. Then  
\[
\frac{a_n}{r_n} < \frac{c}{r_n}
\]
which is bounded from above since $\displaystyle \lim_{n \to \infty} \frac{c}{r_n}=0$.
Moreover
\[
\frac{-a_n}{r_n}= 1-\frac{b_n}{r_n} \leq 1-\frac{c}{r_n}
\]
which again is bounded from above.

Therefore, we are in the situation of Theorem~\ref{thm:main1}.
\end{proof}

\begin{remark} If we do not work with van Hove sequences of intervals, there are many other auto-correlations which are possible. Indeed, let $Q \subseteq \PP$ be any finite set.

Then, $(A_n)_n = ([n!, n!+n] \cup Q)_n$ is a van Hove sequence and $F_n=A_n \cap \PP= Q$. Comparing to Remark \ref{rem:primes_can_have_0_diffraction}, it follows that  
\[
\gamma_{\operatorname{count}}=\frac{1}{\card(Q)}\sum_{p,q \in Q} \delta_{p-q}, \,\qquad \, \reallywidehat{\gamma_{\operatorname{count}}}(x)=\frac{1}{\card(Q)}\sum_{p \in Q} \left|e^{2 \pi i (q-p)x }\right|^2 \,.
\]
\end{remark}

By combining Proposition \ref{prop:no_autocorrelation_shift} and Theorem~\ref{thm:primes_counting_diffraction} we also get the following example:

\begin{example} Let $k$ be a nonzero integer { and consider the set
\[
\PP\cup (k+\PP) = \{p, p+k \;\mid\; p \in \PP \}\, .
\]}
Then, with respect to the van Hove sequence $([-n,n])_n$, $\gamma_{\operatorname{count}}=\delta_0+\frac{1}{2}(\delta_{k}+\delta_{-k}) $ and so
\[
\widehat{\gamma_{\operatorname{count}}}=(1+\cos(2\pi k x)) \lambda \,.
\]
\end{example}

\subsection{Diffraction of prime powers}

Consider now the set 
{\[
\PP_{\operatorname{pow}} := \{ \pm p^n : p \in \PP, n \in \NN \}
\]}
of positive and negative prime powers. The point sets $\PP \subseteq \PP_{\operatorname{pow}} \subseteq \RR$ and $(A_n)_n=([-n,n])_n$ satisfy the conditions of Lemma~\ref{lemma:subset}. 

Indeed, a standard computation (see for example \cite{Sam}) shows that for all $n$ we have 
\begin{align*}
\card(\PP_{\operatorname{pow}} \cap [-n,n])&= 2 \card(\PP_{\operatorname{pow}} \cap [0,n]) = 2( \pi_0(n)+ \pi_0(\sqrt{n})+\pi_0(\sqrt[3]{n})+ \ldots +\pi_0(\sqrt[m]{n})) \,.
\end{align*}
where $m=\lfloor \log_2(n) \rfloor+1$ since $n < 2^{m+1}$.
This implies that 
\begin{align*}
0 \leq  \frac{\card(\PP_{\operatorname{pow}} \cap [-n,n])}{ \card(\PP \cap [-n,n])} -1 &\leq \frac{\pi_0(\sqrt{n})+\pi_0(\sqrt[3]{n})+ \ldots +\pi_0(\sqrt[m]{n})}{ \pi_0(n)} \\
&\leq \frac{\sqrt{n}+\sqrt[3]{n}+ \ldots +\sqrt[m]{n}}{ \pi_0(n)} \leq \frac{(\lfloor \log_2(n) \rfloor+1) \cdot \sqrt{n}}{\pi_0(n)} \,.
\end{align*}
Therefore, by the Prime Number Theorem, we get 
\[
\lim_{n \to \infty} \frac{\card(\PP_{\operatorname{pow}} \cap [-n,n])}{ \card(\PP \cap [-n,n])}=1 \,.
\]

Lemma~\ref{lemma:subset} and Theorem~\ref{thm:main1} then give:
\begin{proposition} Let $(A_n)_n=([a_n,b_n])_n$ be any van Hove sequence of intervals. Assume that there exists some constant $d>0$ such that 
\[
|a_n| < d(b_n-a_n) \,.
\]
Then, with respect to $(A_n)_n$ the set $\PP_{\operatorname{pow}}$ of powers of primes has counting autocorrelation $\gamma_{\operatorname{count}}=\delta_0$ and diffraction $\reallywidehat{\gamma_{\operatorname{count}}}=\lambda$.
\end{proposition}

\subsection{Primes with fixed distance}

We complete the section on primes by discussing the diffraction of twin primes and other similar sets. While we cannot explicitly calculate the diffraction of the twin primes, as the answer would need us first to settle the twin prime conjecture, we can show that the diffraction of the twin primes is different from the diffraction of the primes. 

Fix a positive integer $d \geq 1$. Define
\[
\PP_d:=\{ p, p+d : p, p+d \in \PP  \} \,.
\]
When $d=2$, the set $\PP_2$ is exactly the set of twin primes.

\begin{proposition}\label{Prop:diffTP} Let $d \geq 1$ be any positive integer. Then,
\begin{itemize}
\item[(a)] If $\PP_d$ is non-empty and finite then the autocorrelation $\gamma_{\operatorname{count}}$ exists with respect to $(A_n)_n=([-n,n])_n$ and  
\begin{align*}
\gamma_{\operatorname{count}} &= \frac{1}{\card(\PP_d)} \sum_{p,q \in \PP_d} \delta_{p-q} \neq \delta_0, \\
\reallywidehat{\gamma_{\operatorname{count}}}&=\frac{1}{\card(\PP_d)} \sum_{p,q \in \PP_d} e^{2 \pi i (p-q)x } \neq \lambda. 
\end{align*}
Moreover, 
\[
\frac{1}{2} \leq \gamma_{\operatorname{count}}(\{d \}) 
\]
with equality if and only if $p,p+d, p+2d$ cannot be prime at the same time. 

\item[(b)] Assume that $\PP_d$ is infinite and $d$ is not divisible by $3$. Let $\gamma_{\operatorname{count}}$ be the counting autocorrelation of $\PP_d$ with respect to a subsequence of $([-n,n])_n$. Then, 
\begin{align*}
\gamma_{\operatorname{count}}(\{d\}) &= \frac{1}{2}, \\
\reallywidehat{\gamma_{\operatorname{count}}}& \neq \lambda .
\end{align*}
\end{itemize}
\end{proposition}
\begin{proof}
Let us define
\[
\varLambda := \{ p \in \PP_d : p+d \in \PP_d \} \,.
\]
and note that 
\begin{equation}\label{eq:ppd}
\PP_d = \varLambda \cup (d+\varLambda) \,.
\end{equation}

\textbf{(a)} Since $\PP_d$ is finite, we have 
\[
F_n= \PP_d \cap [-n,n] = \PP_d
\]
for $n$ large enough. The two formulas are then immediate. 

Next, \eqref{eq:ppd} implies that
\[
\card(\PP_d) \leq 2 \card (\varLambda)
\]
with equality if and only if $\varLambda \cap (d+\varLambda) = \emptyset$.
Now, 
\begin{align*}
\gamma_{\operatorname{count}}(\{d\}) &= \frac{1}{\card(\PP_d)} \sum_{p,q \in \PP_d} \delta_{p-q}(d) \\
&= \frac{1}{\card(\PP_d)} \sum_{q,q+d \in \PP_d} 1=  \frac{\card(\varLambda)}{\card(\PP_d)} \geq \frac{1}{2}
\end{align*}
with equality if and only if $\varLambda \cap(p+\varLambda)=\emptyset$. This is in turn equivalent to saying that $p,p+d,p+2d$ are never primes at the same time.

\textbf{(b)}
Next, let us note that if 
\[
p \in \varLambda \cap (d+\varLambda)
\]
then $p,p+d, p+2d \in \PP$. Since $3 \nmid d$, one of $p,p+d, p+2d$ must be divisible by $3$. This shows that 
\[
\varLambda \cap (d+\varLambda) \subseteq \{-3-2d, -3-d, -3, 3, 3-d, 3-2d \}
\]
and hence $\varLambda \cap (d+\varLambda)$ is finite.

A short computation similar to the proof of Proposition~\ref{prop:no_autocorrelation_shift} shows that
\[
\gamma_{\operatorname{count}}(\{d\}) = \frac{1}{2} \,.
\]
The claim follows.
\end{proof}

Since $3,5,7 \in \PP$ we obtain:

\begin{corollary} Let $\gamma_{\operatorname{count}}$ be any counting autocorrelation of the twin primes $\PP_2$ with respect to a subsequence of $([-n,n])_n$. Then, $\gamma_{\operatorname{count}}(\{2\})=\frac{1}{2}$ if and only if there are infinitely many twin primes.
\end{corollary}

\begin{remark} \
\begin{itemize}
\item[(a)] The diffraction of the set $\PP_2$ of twin primes is not $\lambda$.
\item[(b)] When $\PP_d$ is finite then the diffraction is absolutely continuous. It is unclear to us what happens when $\PP_d$ is infinite, for example, for the values of $d$ from \cite{Zhang,Polymath,Maynard}. 

The above result shows that, as long as $d$ is not a multiple of $3$, if $\PP_d$ is infinite, then $\PP_d$ and $\PP$ have different diffraction. It is reasonable to expect that for infinite $\PP_d$, the set $\PP_d$ and $\PP$ have different diffraction even when $3|d$, but the above proof does not work anymore. To make the proof work, one would need to show that the set of primes $p$ for which both $p+d$ and $p+2d$ are also prime has asymptotic density zero inside $\PP_d$, which looks reasonable but extremely hard to show.
\end{itemize}
\end{remark}

\section{Subsets of zero density with non-trivial Bragg spectrum }\label{Sect:last}

All the examples we have seen so far have continuous counting diffraction spectrum. We show below that this is not always the case. The following result allows us to construct sets of density zero whose counting diffraction spectrum can be the density diffraction of any relatively dense set that we want.

\begin{theorem}\label{thm:density_to_counting} Let $\varLambda \subseteq \ZZ$ be relatively dense, and assume that the density autocorrelation $\gamma_{\operatorname{dens}}=\gamma$ of $\varLambda$ exists with respect to $\cA = (A_n)_n=([-n,n])_n$.

For each { $n \geq 3$}, set 
\[
\varGamma_n= n!+ \left( \varLambda \cap [-n,n] \right)
\]
and define
\[
\varGamma:= \bigcup_{ n \geq 3} \varGamma_n,
\]
which is a disjoint union. Then, with respect to $\cA$ we have
\begin{itemize}
    \item[(a)] $\dens(\varGamma)=0$, and
    \item[(b)] $\gamma_{\operatorname{count}}= C \gamma$ where 
\[
C= \frac{1}{\gamma(\{0\})} =\frac{1}{\dens(\varLambda)} \neq 0 \,.
\]
\end{itemize}
\end{theorem}

\begin{proof}
{
Let us note first that for $n \geq 3$ we have 
\[
n! +n < (n+1)!-(n+1)
\]
showing that $\Gamma$ is indeed a disjoint union. 

}
Let $\varLambda_n=\varLambda \cap A_n=\varLambda \cap [-n,n]$, so that $\varGamma_n=n!+\varLambda_n$ and $\card(\varGamma_n)=\card(\varLambda_n)\le 2n+1$. Denote $F_n=\varGamma \cap A_n$.

\bigskip

\textbf{(a)} Let $n\in \NN$, and let $m$ be the unique natural number for which
\[
m!\le n<(m+1)!.
\]
Then
$
F_n\subseteq
\bigcup_{{k=3}}^{m+1} \varGamma_k
$,
and so
\begin{align*}
\frac{\card(F_n)}{2n} &=
\frac{1}{2n}\sum_{{k=3}}^{m+1} \card(\varGamma_k\cap [-n,n])\\&\le
\frac{1}{2n}\sum_{{k=3}}^{m+1}\card(\varGamma_k)  \le
\frac{m(2m+1)}{2n}\le
\frac{m(2m+1)}{2m!}\to 0
\end{align*}
as $m\to \infty$. As $n\to \infty$, we also have $m\to \infty$, and so therefore
\[
\dens_\cA(\varGamma)=
\lim_{n \to \infty}\frac{\card(F_n)}{2n}=0.
\]

\bigskip

\textbf{(b)} Let $t\in \ZZ$. For readability, introduce the notation
\begin{align*}
\ell_n &= \card(\varLambda_n)=\card(\varGamma_n),\quad\text{and} \\
c_n &=
\card(\varLambda_n\cap (-t+\varLambda)).
\end{align*}
Then, by Lemma \ref{lem:gammadens}, we know that
\begin{align*}
\gamma(\{0\})=
\dens_\cA(\varLambda) =
\lim_{n \to \infty}\frac{\ell_n}{2n} \ne 0 \quad\text{and}\quad
\gamma(\{t\}) =
\lim_{n \to \infty}\frac{c_n}{2n}.
\end{align*}
Our goal is to show that
\[
\lim_{n \to \infty}\frac{\card(F_n\cap(-t+\varGamma))}{\card(F_n)} =
\frac{\gamma(\{t\})}{\gamma(\{0\})} =
\lim_{n \to \infty}\frac{c_n}{\ell_n},
\]
and then the result follows from Lemma \ref{lem:gammacount}. We will accomplish this by showing that $\frac{\card(F_n\cap(-t+\varGamma))}{\card(F_n)}$ is asymptotically a ratio of the C\'{e}saro means for $(c_n)$ and $(\ell_n)$.

As in the proof of (a), let $n\in \NN$ and let $m\in \NN$ be the unique number satisfying $m!\le n\le (m+1)!$. Then
\[
\bigcup_{{k=3}}^m \varGamma_k \subseteq F_n\subseteq \bigcup_{{k=3}}^{m+1} \varGamma_k,
\]
with these unions disjoint, and so
\[
\sum_{{k=3}}^m \ell_k \le
\card(F_n)\le
\sum_{{k=3}}^{m+1} \ell_k.
\]

Eventually, as $k$ becomes large, $k!$ is large enough that $t+\varGamma_k$ intersects $\varGamma$ only in $\varGamma_k$. That is, for some constant $N$ depending only on $t$, if $k\ge N$, then
\[
\varGamma_k\cap (-t+\varGamma)=
\varGamma_k\cap (-t+\varGamma_k),
\]
which has cardinality $
\card(\varGamma_k\cap (-t+\varGamma_k)) =
\card(\varLambda_k\cap (-t+\varLambda_k))=c_k.
$
(In fact, we could take $N=|t|$, though a smaller number will typically work.)
Now, assume that $n$ is large enough that $m\ge N$, and estimate
\begin{align*}
\card(F_n\cap (-t+\varGamma)) &\ge
\sum_{{k=3}}^m \card(\varGamma_k\cap (-t+\varGamma)) \\ &\ge 
\sum_{k=N}^m \card(\varGamma_k\cap (-t+\varGamma)) \\ &=
\sum_{k=N}^m c_k =
\sum_{{k=3}}^{m} c_k-C,
\end{align*}
where $C=\sum_{{k=3}}^{N-1}c_k$ is a constant not depending on $n$ (only on $\varLambda$ and $t$). And, for an upper bound,
\begin{align*}
\card(F_n\cap (-t+\varGamma)) &\le
\sum_{{k=3}}^{m+1} \card(\varGamma_k\cap (-t+\varGamma)) \\ &=
\sum_{{k=3}}^{N-1}\card(\varGamma_k\cap(-t+\varGamma)) +
\sum_{k=N}^{m+1} \card(\varGamma_k\cap(-t+\varGamma)) \\ &\le
\sum_{{k=3}}^{N-1}\card(\varGamma_k) +
\sum_{k=N}^{m+1} \card(\varGamma_k\cap(-t+\varGamma_k)) \\ &=
\sum_{{k=3}}^{N-1} \ell_k+\sum_{k=N}^{m+1} c_k \;\le\;
\sum_{{k=3}}^{m+1}c_k+D,
\end{align*}
where $D=\sum_{{k=3}}^{N-1}(\ell_k-c_k)$ is a constant not depending on $n$. Combining all of our bounds so far shows that for large enough $n$,
\begin{equation}\label{eq:F_n_bounds}
\frac{\sum_{{k=3}}^m c_k-C}{\sum_{{k=3}}^{m+1}\ell_k} \le
\frac{\card(F_n\cap (-t+\varGamma))}{\card(F_n)}\le
\frac{\sum_{{k=3}}^{m+1}c_k+D}{\sum_{{k=3}}^m \ell_k}.
\end{equation}

Now, because $\varLambda_n$ is relatively dense, there must be a constant $R>0$ so that
\[
Rn\le
\ell_n \le
2n+1
\]
for all $n$. So, 
\begin{align*}
\frac{\sum_{{k=3}}^{m+1}\ell_k}{\sum_{{k=3}}^m \ell_k} -1=
\frac{\ell_{m+1}}{\sum_{{k=3}}^m\ell_k} \le\frac{2m+1}{\sum_{{k=3}}^m Rk}  &=
\frac{2m+1}{R\frac{m(m+1)}{2}-R} \to 0
\end{align*}
as $m\to \infty$, so
\begin{equation}\label{eq:sum_ell_asymptotic}
\lim_{m\to \infty}\frac{\sum_{{k=3}}^{m+1}\ell_k}{\sum_{{k=3}}^m \ell_k} =1.
\end{equation}

Now, as $n\to \infty$, $m\to \infty$, Equation \eqref{eq:sum_ell_asymptotic} gives
\begin{align*}
\lim_{m\to \infty}\frac{\sum_{{k=3}}^m c_k-C}{\sum_{{k=3}}^{m+1}\ell_k} &=
\lim_{m\to \infty}\frac{\sum_{{k=3}}^m c_k-C}{\sum_{{k=3}}^{m}\ell_k} =
\lim_{m\to\infty}\frac{c_m}{\ell_m},
\end{align*}
by the Stolz-C\'{e}saro Theorem. Similarly,
\begin{align*}
\lim_{m\to\infty}\frac{\sum_{{k=3}}^{m+1}c_k+D}{\sum_{{k=3}}^m\ell_k} &=
\lim_{m\to\infty}\frac{\sum_{{k=3}}^{m+1}c_k+D}{\sum_{{k=3}}^{m+1}\ell_k} =
\lim_{m\to\infty}\frac{c_{m+1}}{\ell_{m+1}}=
\lim_{m\to\infty}\frac{c_m}{\ell_m}.
\end{align*}
Now, the Squeeze Theorem applied to Equation \eqref{eq:F_n_bounds} yields
\[
\eta_{\operatorname{count}}(t) =\lim_{n\to\infty}\frac{\card(F_n\cap (-t+\varGamma))}{\card(F_n)} =
\lim_{m\to\infty}\frac{c_m}{\ell_m} =
\frac{\gamma(\{t\})}{\gamma(\{0\})}.
\]
Lemma \ref{lem:gammacount} now shows that $\gamma_{\operatorname{count}}$ exists for $\varGamma$ with respect to $\cA$, and it is $\gamma_{\operatorname{count}} = \frac{1}{\gamma(\{0\})}\gamma$.
\end{proof}

\begin{remark} The conclusion of Theorem~\ref{thm:density_to_counting} holds if in the definition of $\varGamma_n$, the sequence $(n!)_n$ is replaced by any increasing sequence $(a_n)_n$ of natural numbers with the property that 
\[
\lim_{n \to \infty} \frac{a_{n+1}-a_n}{n^2}=\infty \,.
\]
The proof in this more general situation is a straightforward modification of the proof of Theorem \ref{thm:density_to_counting}.
\end{remark}

By applying this result to various known examples, we get a list of subsets of $\ZZ$ of density zero covering most spectral types. Note that the primes already provide an example of a set of density zero with a purely absolutely continuous spectrum; the only spectral type currently missing is purely singular continuous. Such an example cannot be produced by the method of this section, as a relatively dense subset of $\ZZ$ always has a trivial Bragg peak at the origin.

\begin{example} Let $\varLambda =\ZZ$ and let 
\[
\varGamma= \bigcup_{{k=3}}^\infty \{ n!-n,n!-n+1, \ldots, n!+n \} \,.
\]
Then $\varGamma$ has 
\begin{align*}
\gamma_{\operatorname{count}}&= \delta_{\ZZ} \quad\text{and} \\
\widehat{\gamma_{\operatorname{count}}}&= \delta_{\ZZ},
\end{align*}
which are the density autocorrelation and diffraction of $\ZZ$, respectively.
\end{example}


\begin{example}\label{ex4.6} Let $\varLambda$ be the positions of a's in the Thue--Morse substitution (see \cite[Sect.~4.6]{TAO}) and let  
\[
\varGamma= \bigcup_{{k=3}}^\infty \; n!+ (\varLambda \cap [-n,n]).
\]
{ Let $\gamma_{\operatorname{TM}}$ and $\omega_{\operatorname{TM}}$ be the density autoorrelation and diffraction of the $\pm 1$ weighted Thue--Morse comb. Then, by \cite[Thm.~10.1]{TAO}, $\omega$ is a $\ZZ$-periodic purely singular continuous measure.

Next, by \cite[Rem.~10.3]{TAO} we have 
\[
\gamma_{\operatorname{dens}}= \frac{1}{4}\delta_{\ZZ}+\frac{1}{4}\gamma_{\operatorname{TM}}\,.
\]
since the a's in the Thue--Morse comb have density $\frac{1}{2}$ we have $\gamma_{dens}(\{0\}) = \frac{1}{2}$, and hence
\begin{align*}
\gamma_{\operatorname{count}}&= \frac{1}{2}  \delta_{\ZZ}+\frac{1}{2}\gamma_{\operatorname{TM}} \\
\widehat{\gamma_{\operatorname{count}}}&= \frac{1}{2}\delta_{\ZZ}+ \frac{1}{2}\omega_{\operatorname{TM}} \,.
\end{align*}}

\end{example}

\begin{example} Let $\varLambda$ be the positions of a's in the {binary} Rudin--Shapiro substitution (see \cite[Sect.~4.7]{TAO}) and let  
\[
\varGamma= \bigcup_{{k=3}}^\infty \; n!+ (\varLambda \cap [-n,n])
\]
Then{, similar to the Thue--Morse calculations above,
\begin{align*}
\gamma_{\operatorname{count}}&= \frac{1}{2}\delta_{\ZZ}+ \frac{1}{2}\delta_0  \\
\widehat{\gamma_{\operatorname{count}}}&= \frac{1}{2}\delta_{\ZZ}+ \frac{1}{2}\lambda
\end{align*}}
where $\lambda$ is the Lebesgue measure on $\RR$ (see \cite[Remark 10.5]{TAO}). 

{ Note here that 
\[
\gamma_{\operatorname{count}}(\{0\})=\frac{1}{2}+\frac{1}{2}
\]
as Lemma~\ref{lem:ac-at-0}(b) implies.
}
\end{example}

\begin{example}  Let $\varLambda'$ be the { position of $a$'s in the Thue--Morse substitution}, and let $\varLambda$ be a generic element for the Bernoulisation process on $\varLambda'$ with $p=\frac{1}{2}$ \cite[Subsection 11.2.2]{TAO}. Then, $\varLambda \subseteq \ZZ$ almost surely has diffraction 
\[
\widehat{\gamma_{\operatorname{dens}}}= {\frac{1}{16} \delta_{\ZZ} + \frac{1}{16} \omega + \frac{1}{4}} \lambda
\]
where $\omega$ is the singular continuous measure from Example~\ref{ex4.6} (see \cite[Remark 11.3]{TAO}).
Thm.~\ref{thm:density_to_counting} then produces a set $\varGamma$ of density zero with diffraction spectrum containing all three components.

\end{example}

\section{Embedding point sets in higher dimension}\label{Sect:lastlast}

Let us now briefly look at a different situation where sets of density zero appear naturally.
Let $\varLambda \subseteq \RR^d$ be an uniformly discrete point set. Then we can identify $\varLambda$ with 
$\varLambda \times \{ 0 \} \subseteq \RR^{d+m}$.
Using the standard van Hove sequence $(A_n)_n=([-n,n]^{d+m})_n$ we have 
\[
\dens(\varLambda)=0 \,.
\]
This means that the classical definition of (density) diffraction would always give a diffraction 
\[
\reallywidehat{\gamma_{\operatorname{dens}}}=0 \,.
\]
But now consider an (idealized) multiple-slit interference experiment. This can simply be described by punching  tiny holes into a wall at locations $\varLambda \subseteq \RR$ and doing a 2-dimensional diffraction experiment. The outcome of the experiment is usually a mixture of lights and shadows on a wall, which are periodic in one direction and change in the opposite direction.

We will see below that the counting autocorrelation precisely describes this phenomena. Recall first that given two measures $\mu,\nu$ on $G$ and $H$ respectively, there exists a unique product measure $\mu \times  \nu$ on $G \times H$ with the property that 
\[
\int_{G \times H} f(x) g(y) \dd \mu \times  \nu (x,y) = \left( \int_{G} f(x) \dd \mu(x) \right) \left( \int_{H} g(y) \dd \nu(y) \right), \qquad \forall f \in \Cc(G), g \in \Cc(H) \,.
\]
With this notation, we trivially have
\[
\delta_x \times \delta_y= \delta_{(x,y)} \,.
\]

Let us now show how the counting autocorrelation behaves when we increase the dimension of the underlying space. 

\begin{theorem} Let $\varLambda \subseteq G$ be any point set, and assume that the autocorrelation $\gamma_{\operatorname{count},\varLambda}$ exists with respect to some van Hove sequence $(A_n)_n$. 

Let $H$ be any second countable LCAG and $(B_n)_n$ any van Hove sequence in $H$, such that $0 \in B_n$ for all $n$. Let $C_n =A_n \times B_n \subseteq G \times H$.

Define
\[
\varGamma = \varLambda \times \{ 0 \} \subseteq G \times H \,.
\]
Then, the counting autocorrelation $\gamma_{\operatorname{count},\varGamma}$ of $\varGamma \subseteq G \times H$ exists with respect to $(C_n)_n$ and 
\[
\gamma_{\operatorname{count},\varGamma} = \gamma_{\operatorname{count},\varLambda} \times \delta_0 \,.
\]
In particular, 
\[
\reallywidehat{\gamma_{\operatorname{count},\varGamma}} = \reallywidehat{\gamma_{\operatorname{count},\varLambda}} \times \theta_{\widehat{H}} \,.
\]
where $\theta_{ \widehat{H}}$ is a Haar measure on $\widehat{H}$.
\end{theorem}
\begin{proof}
Define 
\begin{align*}
F_n&:= \varLambda \cap A_n\, , \\
E_n&= \varGamma \cap C_n \, .
\end{align*}
Then, since $0 \in B_n$, we have 
\[
E_n = F_n \times \{ 0 \}
\]
and hence 
\[
\delta_{E_n} = \delta_{F_n} \times \delta_0 \,.
\]
It follows immediately that 
\begin{align*}
\gamma_{E_n}&= \frac{1}{\card(E_n)} \delta_{E_n} \ast \widetilde{\delta_{E_n}}= \frac{1}{\card(F_n)} \delta_{F_n \times \{ 0 \}} \ast \widetilde{\delta_{F_n \times \{ 0 \}}} \\
&= \left( \frac{1}{\card(F_n)} \delta_{F_n } \ast \widetilde{\delta_{F_n}} \right) \times \delta_0 = \gamma^{}_{F^{}_n} \times \delta_0 \,,
\end{align*}
from which the rest of the proof follows.
\end{proof}

\begin{example} Consider $\ZZ \times \{ 0 \} \subseteq \RR^d$. Then, its counting diffraction is 
\[
\reallywidehat{\gamma_{\operatorname{count}}} = \delta_{\ZZ} \times \lambda \,.
\]
\end{example}

\begin{example} Let $\varLambda \subseteq \RR$ be the Silver mean model set (see \cite[Example~4.5]{TAO}) and let $\varGamma = \varLambda \times \{ 0 \} \subseteq \RR^{1+m}$. Then, the diffraction of $\varGamma$ is 
\[
\reallywidehat{\gamma_{\operatorname{count}}} = \sigma \times \lambda \,,
\]
where 
\[
\sigma=  \sum_{k \in \frac{1}{2}\ZZ+\frac{1}{2\sqrt{2}}\ZZ} A_k \delta_k
\]
is the pure point measure from \cite[Theorem~9.3]{TAO}.
\end{example}

\section*{Acknowledgements}
A.H. was partially supported by the NSERC Discovery grant 2024-03883, C.R. was supported by the NSERC Discovery grant 2019-05430, and N.S. was supported by the NSERC Discovery grants 2020-00038 and 2024-04853. { The authors would like to thank the anonymous reviewer for some suggestions which improved the quality of the manuscript.}

\end{document}